\documentclass[11pt]{amsart}
\usepackage{a4wide}
\usepackage{cite}
\usepackage{amsmath, amsfonts, amssymb,graphicx}
\usepackage{mathrsfs,url}

\author{Manuel Bodirsky}
    \address{Laboratoire d'Informatique  (LIX), CNRS UMR 7161\\
    \'{E}cole Polytechnique \\91128 Palaiseau\\
    France}
    \email{bodirsky@lix.polytechnique.fr}
    \urladdr{http://www.lix.polytechnique.fr/~bodirsky/}
\author{Michael Pinsker}
    \address{\'{E}quipe de Logique Mathématique\\ Universit\'{e} Denis Diderot - Paris 7\\
	UFR de Math\'{e}matiques\\
	75205 Paris Cedex 13, France}
    \email{marula@gmx.at}
    \urladdr{http://dmg.tuwien.ac.at/pinsker/}
    \thanks{The first author has received funding from the European Research Council under the European Community's Seventh Framework Programme (FP7/2007-2013 Grant Agreement no. 257039). The second author is grateful for support through an APART-fellowship of the Austrian Academy of Sciences.}

\title[Reducts of Ramsey structures]{Reducts of Ramsey structures}

\newcommand{\Cycl}{\text{\it Cycl}}
\newcommand{\Betw}{\text{\it Betw}}
\newcommand{\btw}{\text{\it Betw}}
\newcommand{\sep}{\text{\it Sep}}
\newcommand{\Fresse}{Fra\"{i}ss\'{e}}
\DeclareMathOperator{\NAE}{NAE}
\DeclareMathOperator{\OIT}{1IN3}

\newcommand{\ignore}[1]{}

\newcommand{\cl}[1]{\langle #1 \rangle}
\newcommand{\dlo}{(\mathbb{Q};<)}
\newcommand{\To}{\rightarrow}
\newcommand{\nin}{\notin}

\newcommand{\mult}{\times}

\DeclareMathOperator{\Expr}{Expr}

\newcommand{\Exprpp}{\Expr_{pp}}
\newcommand{\Exprep}{\Expr_{ep}}

\DeclareMathOperator{\pp}{pp}

\DeclareMathOperator{\lex}{lex}
\DeclareMathOperator{\Csp}{CSP}
\DeclareMathOperator{\id}{id}
\DeclareMathOperator{\tp}{tp}
\DeclareMathOperator{\sw}{sw}
\DeclareMathOperator{\Aut}{Aut}
\DeclareMathOperator{\End}{End}
\DeclareMathOperator{\Pol}{Pol}
\DeclareMathOperator{\Inv}{Inv}

\newcommand{\C}{\mathcal C}
\newcommand{\D}{\mathcal D}

\newcommand{\F}{\mathcal F}
\newcommand{\G}{\mathcal G}
\renewcommand{\H}{\mathcal H}

\newcommand{\M}{\mathcal M}
\newcommand{\N}{\mathcal N}

\renewcommand{\P}{\mathcal P}

\renewcommand{\S}{\mathcal S}

\theoremstyle{plain}

    \newtheorem{thm}{Theorem}
    \newtheorem{theorem}[thm]{Theorem}

    \newtheorem{lem}[thm]{Lemma}
    \newtheorem{lemma}[thm]{Lemma}

    \newtheorem{prop}[thm]{Proposition}

    \newtheorem{cor}[thm]{Corollary}
    \newtheorem{corollary}[thm]{Corollary}

    \newtheorem{conj}[thm]{Conjecture}

\theoremstyle{definition}

    \newtheorem{defn}[thm]{Definition}
    \newtheorem{definition}[thm]{Definition}

\begin{document}

\maketitle

\begin{abstract}
One way of studying a relational structure is to investigate functions
which are related to that structure and which leave certain aspects of
the structure invariant. Examples are the automorphism group, the
self-embedding monoid, the endomorphism monoid, or the polymorphism
clone of a structure.
Such functions can be particularly well understood when the relational
structure is countably infinite and has a first-order definition in another relational structure which has a finite language, is totally ordered and homogeneous, and has the Ramsey property. This is because
in this situation, Ramsey theory provides the combinatorial tool for
analyzing these functions -- in a certain sense, it allows to
represent such functions by functions on finite sets.

This is a survey of results in model theory and theoretical computer
science obtained recently by the authors in this context. In model
theory, we approach the problem of classifying the reducts of countably infinite ordered homogeneous 
Ramsey structures in a finite language, and certain decidability questions connected with
such reducts. In theoretical computer science, we use the same
combinatorial methods in order to classify the computational
complexity for various classes of infinite-domain constraint
satisfaction problems. While the first set of applications is
obviously of an infinitary character, the second set concerns
genuinely finitary problems -- their unifying feature is that the same
tools from Ramsey theory are used in their solution.

\end{abstract}
\tableofcontents


\section{Introduction}
\begin{flushright}
\emph{``I prefer finite mathematics much more than infinite mathematics. I think that it is much more natural, much more appealing and the theory is much more beautiful. It is very concrete. It is something that you can touch and something you can feel and something to relate to.
Infinity mathematics, to me, is something that is meaningless, because it is abstract nonsense.''}

(Doron Zeilberger, February 2010)
\vspace{1cm}


\emph{``To the person who does deny infinity and says that it doesn't exist, I feel sorry for them, I don't see how such view enriches the world. Infinity may be does not exist, but it is a beautiful subject. I can say that the stars do not exist and always look down, but then I don't see the beauty of the stars. Until one has a real reason to doubt the existence of mathematical infinity, I just don't see the point.''}

(Hugh Woodin, February 2010)
\end{flushright}

\vspace{1cm}

Sometimes, infinite mathematics is not just beautiful, but also
\emph{useful}, even when one is ultimately interested in finite
mathematics. A fascinating example of this type of mathematics
is the recent theorem by Kechris, Pestov, and Todorcevic~\cite{Topo-Dynamics}, which
links Ramsey classes and topological dynamics.
A class of finite structures $\mathcal C$ closed under isomorphisms, induced substructures, and with the joint embedding property (see~\cite{HodgesLong}) is called a \emph{Ramsey class}~\cite{RamseyClasses,NesetrilSurvey} (or \emph{has the Ramsey property}) if for all $P,H \in \mathcal C$
and every $k \geq 2$ there is a $S \in \mathcal C$ such that
for every coloring of the copies of $P$ in $S$ with $k$ colors there is a copy $H'$ of $H$ in $\mathcal C$ such that all copies of $P$ in $H'$ have the same color. This is a very strong requirement --- and
certainly from the finite world.
Proving that a class has the Ramsey property
can be difficult~\cite{NesetrilSurvey}, and Ramsey theory rather provides a tool box than a theory to answer this question.

Kechris, Pestov, and Todorcevic~\cite{Topo-Dynamics} provide a characterization
of such classes in topological dynamics, connecting Ramsey classes
with \emph{extreme amenability} in (infinite) group theory, a concept from the 1960s~\cite{Granirer}.
The result can be used in two directions. One can use it to translate
deep existing Ramsey results into proofs of extreme amenability of topological groups (and this is the
main focus of the already cited article~\cite{Topo-Dynamics}). One can also use it in the other direction
to obtain a more systematic understanding of Ramsey classes.
A key insight for this direction is the result of Ne\v{s}et\v{r}il (see~\cite{RamseyClasses}) which says
that Ramsey classes $\mathcal C$
have the \emph{amalgamation property}. Hence,
by \Fresse's theorem, there exists a countably infinite homogeneous and $\omega$-categorical structure $\Gamma$ such that a finite structure is from $\mathcal C$ if and only if it embeds into $\Gamma$.
The structure $\Gamma$ is unique up to isomorphism, and is called the
\emph{\Fresse~limit} of $\mathcal C$.
Now let $\mathcal D$ be any amalgamation class whose \Fresse~limit 
$\Delta$ is bi-interpretable with $\Gamma$. By the theorem of Ahlbrandt and Ziegler~\cite{AhlbrandtZiegler}, two $\omega$-categorical structures are first-order bi-interpretable if and only if their automorphism groups are isomorphic as \emph{(abstract) topological groups}.
 In addition, the above-mentioned result from~\cite{Topo-Dynamics} shows that whether
or not $\mathcal D$ is a Ramsey class only depends on the automorphism
group $\text{Aut}(\Delta)$ of $\Delta$;
in fact, and much more interestingly, it only depends on $\text{Aut}(\Delta)$
viewed as a topological group (which has cardinality $2^\omega$).
From this we immediately get our first example where~\cite{Topo-Dynamics}
is used in the second direction, with massive consequences for finite structures: the Ramsey property is preserved under first-order bi-interpretations.
We will see another statement of this type (Proposition~\ref{prop:addingConstantsPreservesRamsey}) and more concrete applications of such statements later (in Section~\ref{sect:minimalfunctions}, Section~\ref{sect:interpret}, and Section~\ref{sect:csp}).
\vspace{.5cm}

\paragraph{\bf Constraint Satisfaction.}
Our next example where infinite mathematics is a powerful tool comes
from (finite) computer science.
A \emph{constraint satisfaction problem} is a computational problem
where we are given a set of variables and a set of constraints
on those variables, and where the task is to decide whether there
is an assignment of values to the variables that satisfies all constraints. Computational problems of this type appear in many
areas of computer science, for example in artificial intelligence,
computer algebra, scheduling, computational linguistics, and computational biology.

As an example, consider the \textsc{Betweenness} problem.
The input to this problem consists of a finite set of variables $V$,
and a finite set of triples of the form $(x,y,z)$ where $x,y,z \in V$.
The task is to find an ordering $<$ on $V$ such that for each of the given triples $(x,y,z)$ we have either $x<y<z$ or $z<y<x$.
It is well-known that this problem is NP-complete~\cite{Opatrny,GareyJohnson}, and that we therefore cannot expect to find
a polynomial-time algorithm that solves it.
In contrast, when we want to find an ordering $<$ on $V$ such
that for each of the given triples $(x,y,z)$ we have $x<y$ or $x<z$,
then the corresponding problem can be solved in polynomial time.

Many constraint satisfaction problems can be modeled formally as follows. Let $\Gamma$ be a structure with a finite relational signature.
Then the \emph{constraint satisfaction problem for $\Gamma$}, denoted by $\Csp(\Gamma)$, is the problem of deciding whether a given primitive positive sentence $\phi$ is true in $\Gamma$.
By choosing $\Gamma$ appropriately, many problems in the above
mentioned application areas can be expressed as $\Csp(\Gamma)$.
The \textsc{Betweenness} problem, for instance, can be modeled as
$\Csp(({\mathbb Q}; \btw))$ where $\mathbb Q$ are the rational numbers and $\btw = \{(x,y,z) \in {\mathbb Q}^3 \; | \; x<y<z \vee z<y<x \}$.

Note that even though the structure $\Gamma$ might be infinite,
the problem $\Csp(\Gamma)$ is always a well-defined and
\emph{discrete}
problem. Since the signature of $\Gamma$ is finite,
the complexity of $\Csp(\Gamma)$ is independent of the representation of the relation symbols of $\Gamma$ in input
instances of $\Csp(\Gamma)$. The task is to decide
whether there \emph{exists} an assignment to the variables of a given instance, and
we do not have to exhibit such a solution. Therefore, the computational problems under consideration are finitistic and concrete even when the domain of $\Gamma$ is, say, the real numbers.

There are many reasons to formulate a discrete problem as $\Csp(\Gamma)$ for an infinite structure $\Gamma$.
The advantages of such a formulation are most striking when
$\Gamma$ can be chosen to be $\omega$-categorical.
In this case, the computational complexity of $\Csp(\Gamma)$ is
fully captured by the \emph{polymorphism clone} of $\Gamma$;
the polymorphism clone can be seen as a higher-dimensional
generalization of the automorphism group of $\Gamma$.
When studying polymorphism clones, we can
apply techniques from universal algebra, and, as we will see here,
from Ramsey theory to obtain results about the computational complexity of $\Csp(\Gamma)$.

\vspace{.5cm}
\paragraph{\bf Contributions and Outline.}
In this article we give a survey presentation of 
a technique how to apply Ramsey theory
when studying automorphism groups, endomorphism monoids, and polymorphism
clones of countably infinite structures with a first-order definition in an ordered homogeneous
Ramsey structure in a finite language -- such structures are always $\omega$-categorical.
We
present applications of this technique in two fields. Let $\Delta$ be
a countable structure with a first-order definition in an ordered homogeneous Ramsey
structure in a finite language. In model theory, our technique can be used to classify the
set of all structures $\Gamma$ that are first-order definable in
$\Delta$.
In constraint satisfaction, it can be used to obtain a complete
complexity classification for the class of all problems CSP$(\Gamma)$
where
$\Gamma$ is first-order definable in $\Delta$.
We demonstrate this for $\Delta = ({\mathbb Q}; <)$, and for
$\Delta = (V;E)$, the countably infinite random graph.

\section{Reducts}\label{sect:reducts}
One way to classify relational structures on a fixed domain is by identifying two structures when they \emph{define} one another. The term ``define'' will classically stand for ``first-order define'', i.e., a structure $\Gamma_1$ has a first-order definition in a structure $\Gamma_2$ on the same domain iff all relations of $\Gamma_1$ can be defined by a first-order formula over $\Gamma_2$. When $\Gamma_1$ has a first-order definition in $\Gamma_2$ and vice-versa, then two structures are considered equivalent \emph{up to first-order interdefinability}.

Depending on the application, other notions of definability might be suitable; such notions include syntactic restrictions of first-order definability. In this paper, besides first-order definability, we will consider the notions of \emph{existential positive definability} and \emph{primitive positive definability}; in particular, we will explain the importance of the latter notion in theoretical computer science in Section~\ref{sect:csp}.

The structures which we consider in this article will all be countably infinite, and we will henceforth assume this property without further mentioning it. A structure is called \emph{$\omega$-categorical} if all countable models of its first-order theory are isomorphic. We are interested in the situation where all structures to be classified are \emph{reducts} of a single countable $\omega$-categorical structure in the following sense (which differs from the standard definition of a reduct and morally follows e.g.~\cite{RandomReducts}).

\begin{definition}
    Let $\Delta$ be a structure. A \emph{reduct} of $\Delta$ is a structure with the same domain as $\Delta$ all of whose relations can be defined by a first-order formula in $\Delta$.
\end{definition}

When all structures under consideration are reducts of a countably infinite base structure $\Delta$ which is $\omega$-categorical, then there are natural ways of obtaining classifications up to first-order, existential positive, or primitive positive interdefinability by means of certain sets of functions. In this section, we explain these ways, and give some examples of classifications that have been obtained in the past. In the following sections, we then observe that these results have actually been obtained in a more specific context than $\omega$-categoricity, namely, where the structures are reducts of an \emph{ordered Ramsey} structure $\Delta$ which has a finite relational signature and which is \emph{homogeneous} in the sense that every isomorphism between finite induced substructures of $\Delta$ can be extended to an automorphism of $\Delta$. 
We further develop a general framework for proving such results in this context.

We start with first-order definability. Consider the assignment that sends every structure $\Gamma$ with domain $D$ to its automorphism group $\Aut(\Gamma)$. Automorphism groups are closed sets in the convergence topology of all permutations on $D$, and conversely, every closed permutation group on $D$ is the automorphism group of a relational structure with domain $D$. The closed permutation groups on $D$ form a complete lattice, where the meet of a set of groups is given by their intersection. Similarly, the set of those relational structures on $D$ which are \emph{first-order closed}, i.e., which contain all relations which they define by a first-order formula, forms a lattice, where the meet of a set $S$ of such structures is the structure which has those relations that are contained in all structures in $S$. Now when $\Gamma$ is a countable $\omega$-categorical structure, then it follows from the proof of the theorem of Ryll-Nardzewki (see~\cite{HodgesLong}) that its automorphism group $\Aut(\Gamma)$ still has the first-order theory of $\Gamma$ encoded in it. And indeed we can, up to first-order interdefinability, recover $\Gamma$ from its automorphism group as follows: For a set $\F$ of finitary functions on $D$, let $\Inv(\F)$ be the structure on $D$ which has those relations $R$ which are \emph{invariant under $\F$}, i.e., those relations that
contain $f(r_1,\ldots,r_n)$ (calculated componentwise) whenever $f\in \F$ and $r_1,\ldots,r_n\in R$.

\begin{thm}\label{thm:groups-fo}
    Let $\Delta$ be $\omega$-categorical. Then the mapping $\Gamma \mapsto \Aut(\Gamma)$ is an antiisomorphism between the lattice of first-order closed reducts of $\Delta$ and the lattice of closed permutation groups containing $\Aut(\Delta)$. 
    The inverse mapping is given by $\G\mapsto\Inv(\G)$.
\end{thm}

This connection between closed permutation groups and first-order definability has been exploited several times in the past in order to obtain complete classifications of reducts of $\omega$-categorical structures. For example, let $\Delta$ be the order of the rational numbers -- we write $\Delta=\dlo$. Then it has been shown in~\cite{Cameron5} that there are exactly five reducts of $\Delta$, up to first-order interdefinability, which we will define in the following.

On the permutation side, let $\leftrightarrow$ be the function that sends every $x\in \mathbb{Q}$ to $-x$. For our purposes, we can equivalently  choose $\leftrightarrow$ to be any permutation that inverts the order $<$ on $\mathbb{Q}$. For any fixed irrational real number $\alpha$, let $\circlearrowright$ be any permutation on $\mathbb{Q}$ with the property that $x< y< \alpha < u< v$ implies ${\circlearrowright}{(u)}< {\circlearrowright}(v)<{\circlearrowright}(x)<{\circlearrowright}(y)$, for all $x,y,u,v\in \mathbb{Q}$. We will consider closed groups \emph{generated} by these permutations: For a set of permutations $\F$ and a closed permutation group $\G$, we say that $\F$ generates $\G$ iff $\G$ is the smallest closed group containing $\F$.

On the relational side, for $x_1,\ldots,x_n\in \mathbb{Q}$ write $\overrightarrow{x_1\ldots x_n}$ when $x_1<\ldots<x_n$. Then we define a ternary relation $\Betw$ on $\mathbb{Q}$ by $\Betw:=\{(x,y,z)\in \mathbb{Q}^{3}\;|\; \overrightarrow{xyz}\vee \overrightarrow{zyx}\}$. Define another ternary relation $\Cycl$ by $\Cycl:=\{(x,y,z) \in\mathbb{Q}^{3}\;|\; \overrightarrow{xyz}\vee \overrightarrow{yzx}\vee \overrightarrow{zxy}\}$. Finally, define a $4$-ary relation $\sep$ by 
\begin{align*}
 \{(x_1,y_1,x_2,y_2) \in \mathbb{Q}^{4} \;  | \; &
 \overrightarrow{x_1x_2y_1y_2} \vee \overrightarrow{x_1y_2y_1x_2} \vee
 \overrightarrow{y_1x_2x_1y_2} \vee \overrightarrow{y_1y_2x_1x_2} \\
 \vee & \overrightarrow{x_2x_1y_2y_1}  \vee
 \overrightarrow{x_2y_1y_2x_1} \vee \overrightarrow{y_2x_1x_2y_1} \vee \overrightarrow{y_2y_1x_2x_1}\} \,.
 \end{align*}

\begin{thm}[Cameron~\cite{Cameron5}]\label{thm:cameron5}
    Let $\Gamma$ be a reduct of $\dlo$. Then exactly one of the following holds:
    \begin{itemize}
        \item $\Gamma$ is first-order interdefinable with $\dlo$; equivalently, $\Aut(\Gamma)={\Aut(\dlo)}$.
        \item $\Gamma$ is first-order interdefinable with $(\mathbb{Q};\Betw)$; equivalently, $\Aut(\Gamma)$ equals the closed group generated by $\Aut(\dlo)$ and $\leftrightarrow$.
        \item $\Gamma$ is first-order interdefinable with $(\mathbb{Q};\Cycl)$; equivalently, $\Aut(\Gamma)$ equals the closed group generated by $\Aut(\dlo)$ and $\circlearrowright$.
        \item $\Gamma$ is first-order interdefinable with $(\mathbb{Q};\sep)$; equivalently, $\Aut(\Gamma)$ equals the closed group generated by $\Aut(\dlo)$ and $\{\leftrightarrow, \circlearrowright\}$.
        \item $\Gamma$ is first-order interdefinable with $(\mathbb{Q};=)$; equivalently, $\Aut(\Gamma)$ equals the group of all permutations on $\mathbb{Q}$.
    \end{itemize}
\end{thm}

Another instance of the application of Theorem~\ref{thm:groups-fo} in the classification of reducts up to first-order interdefinability has been provided by Thomas~\cite{RandomReducts}. Let $G=(V;E)$ be the random graph, i.e., the up to isomorphism unique countably infinite graph which is homogeneous and which contains all finite graphs as induced subgraphs. It turns out that  up to first-order interdefinability, $G$ has precisely five reducts, too.

On the permutation side, observe that the graph $\bar{G}$ obtained by making two distinct vertices $x,y\in V$ adjacent iff they are not adjacent in $G$ is isomorphic to $G$; let $-$ be any permutation on $V$ witnessing this isomorphism. Moreover, for any fixed vertex $0\in V$, the graph obtained by making all vertices which are adjacent with $0$ non-adjacent with $0$, and all vertices different from $0$ and non-adjacent with $0$ adjacent with $0$, is isomorphic to $G$. Let $\sw$ be any permutation on $V$ witnessing this fact.

On the relational side, define for all $k\geq 2$ a $k$-ary relation $R^{(k)}$ on $V$ by
\begin{align*}
R^{(k)}:=\{(x_1,\ldots,x_k)\;|\;&\text{ all } x_i\text{ are distinct},\\ &\text{ and the number of edges on }\{x_1,\ldots,x_k\}\text{ is odd}\}.
\end{align*}
\begin{thm}[Thomas~\cite{RandomReducts}]\label{thm:thomas5}
    Let $\Gamma$ be a reduct of the random graph $G=(V;E)$. Then exactly one of the following holds:
    \begin{itemize}
        \item $\Gamma$ is first-order interdefinable with $G$; equivalently, $\Aut(\Gamma)=\Aut(G)$.
        \item $\Gamma$ is first-order interdefinable with $(V;R^{(3)})$; equivalently, $\Aut(\Gamma)$ equals the closed group generated by $\Aut(G)$ and $\sw$.
        \item $\Gamma$ is first-order interdefinable with $(V;R^{(4)})$; equivalently, $\Aut(\Gamma)$ equals the closed group generated by $\Aut(G)$ and $-$.
        \item $\Gamma$ is first-order interdefinable with $(V;R^{(5)})$; equivalently, $\Aut(\Gamma)$ equals the closed group generated by $\Aut(G)$ and $\{\sw, -\}$.
        \item $\Gamma$ is first-order interdefinable with $(V;=)$; equivalently, $\Aut(\Gamma)$ equals the group of all permuations on $V$.
    \end{itemize}
\end{thm}

In a similar fashion, the reducts of several prominent $\omega$-categorical structures $\Delta$ have been classified up to first-order interdefinability by finding all closed supergroups of $\Aut(\Delta)$. Examples are:

\begin{itemize}
    \item The countable homogeneous $K_n$-free graph, i.e., the unique countable homogeneous graph which contains precisely those finite graphs which do not contain a clique of size $n$ as induced subgraphs, has 2 reducts up to first-order interdefinability (Thomas~\cite{RandomReducts}), for all $n\geq 3$.
    \item The countable homogeneous $k$-hypergraph has $2^k+1$ reducts up to first-order interdefinability (Thomas~\cite{Thomas96}), for all $k\geq 2$.
    \item The structure $(\mathbb{Q};<,0)$, i.e., the order of the rationals which in addition ``knows'' one of its points, has 116 reducts up to first-order interdefinability (Junker and Ziegler~\cite{JunkerZiegler}).
\end{itemize}

All these examples have in common that the structures have a high degree of symmetry in the sense that they are homogeneous in a finite language -- intuitively, one would expect the automorphism group of such a structure to be rather large. And indeed, Thomas conjectured in~\cite{RandomReducts}:

\begin{conj}[Thomas~\cite{RandomReducts}]\label{conj:thomas}
    Let $\Delta$ be a countable relational structure which is homogeneous in a finite language. Then $\Delta$ has finitely many reducts up to first-order interdefinability.
\end{conj}

It turns out that all the examples above are not only homogeneous in a finite language; in fact, they all have a first-order definition in (in other words: are themselves reducts of) an \emph{ordered Ramsey structure} which is homogeneous in a finite language. Functions on such structures, in particular automorphisms of reducts, can be analyzed by the means of Ramsey theory, and we will outline a general method for classifying the reducts of such structures in Sections~\ref{sect:ramseyclasses} to~\ref{sect:decidability}.

We now turn to analogs of Theorem~\ref{thm:groups-fo} for syntactic restrictions of first-order logic.
A first-order formula is called \emph{existential} iff it is of the form $\exists x_1\ldots\exists x_n.\ \phi$, where $\phi$ is quantifier-free. It is called \emph{existential positive} iff it is existential and does not contain any negations. Now observe that similarly to permutation groups, the \emph{endomorphism monoid} $\End(\Delta)$ of a relational structure $\Delta$ with domain $D$ is always closed in the pointwise convergence topology on the space of all functions from $D$ to $D$, and that every closed transformation monoid $\M$ acting on $D$ is the endomorphism monoid of the structure $\Inv(\M)$, i.e., the structure with domain $D$ which contains those relations which are invariant under all functions in $\M$. Note also that the set of closed transformation monoids on $D$, ordered by inclusion, forms a complete lattice, and that likewise the set of all existential positive closed structures forms a complete lattice. The analog to Theorem~\ref{thm:groups-fo} for existential positive definability is an easy consequence of the homomorphism preservation theorem (see~\cite{HodgesLong}) and goes like this:

\begin{thm}\label{thm:monoids-expos}
    Let $\Delta$ be $\omega$-categorical. Then the mapping $\Gamma\mapsto \End(\Gamma)$ is an antiisomorphism between the lattice of existential positive closed reducts of $\Delta$ and the lattice of closed transformation monoids containing $\Aut(\Delta)$. The inverse mapping is given by $\M\mapsto\Inv(\M)$.
\end{thm}

All the closed monoids containing the group of all permutations on a countably infinite set $D$ (which equals the automorphism group of the empty structure $(D;=)$) have been determined in~\cite{BodChenPinsker}, and their number is countably infinite. Therefore, every structure has infinitely many reducts up to existential positive interdefinability. In general, it will be impossible to determine all of them, but sometimes it is already useful to determine certain closed monoids, as in the following theorem about endomorphism monoids of reducts of the random graph from~\cite{RandomMinOps}. We need the following definitions. Since the random graph $G=(V;E)$ contains all countable graphs, it contains an infinite clique. Let $e_E$ be any injective function from $V$ to $V$ whose image induces such a clique in $G$. Similarly, let $e_N$ be any injection from $V$ to $V$ whose image induces an independent set in $G$.

\begin{thm}[Bodirsky and Pinsker~\cite{RandomMinOps}]\label{thm:randomMinimalMonoids}
    Let $\Gamma$ be a reduct of the random graph $G=(V;E)$. Then at least one of the following holds.
    \begin{itemize}
        \item $\End(\Gamma)$ contains a constant operation.
        \item $\End(\Gamma)$ contains $e_E$.
        \item $\End(\Gamma)$ contains $e_N$.
        \item $\Aut(\Gamma)$ is a dense subset of $\End(\Gamma)$ (equipped with the topology of pointwise
        convergence).
    \end{itemize}
\end{thm}

Theorem~\ref{thm:randomMinimalMonoids} states that for reducts $\Gamma$ of the random graph, either $\End(\Gamma)$ contains a function that destroys all structure of the random graph, or it contains basically no functions except the automorphisms. This has the following non-trivial consequence. A theory $T$ is called
\emph{model-complete} iff every embedding between models of $T$ is
elementary, i.e., preserves all first-order formulas. A structure is said to
be model-complete iff its first-order theory is model-complete.

\begin{cor}[Bodirsky and Pinsker~\cite{RandomMinOps}]\label{cor:mc}
    All reducts of the random graph are model-complete.
\end{cor}
\begin{proof}[Proof]

    It is not hard to see (cf.~\cite{RandomMinOps}) that an $\omega$-categorical structure $\Gamma$ is model-complete if and only if $\Aut(\Gamma)$ is dense in the monoid of self-embeddings of $\Gamma$.
 Now let $\Gamma$ be a reduct of $G$, and let $\M$ be the closed monoid of self-embeddings of $\Gamma$; we will show that $\Aut(\Gamma)$ is dense in $\M$.
    We apply Theorem~\ref{thm:randomMinimalMonoids} to $\M$ (which, as a closed monoid containing $\Aut(G)$, is also an endomorphism monoid of a reduct $\Gamma'$ of $G$). Clearly, $\Gamma'$ and $\Gamma$ have the same automorphisms, namely those permutations in $\M$ whose inverse is also in $\M$. Therefore we are done if the last case of the
    theorem holds. Note that $\M$ cannot contain a constant operation as all its operations are injective. So suppose
    that $\M$ contains $e_N$ -- the argument for $e_E$ is analogous. Let $R$ be any relation of $\Gamma$, and $\phi_R$ be its defining quantifier-free formula; $\phi_R$ exists since $G$ has quantifier-elimination, i.e., every first-order formula over $G$ is equivalent to a quantifier-free formula. Let $\psi_R$ be the formula obtained by replacing all occurrences of $E$ by \emph{false}; so $\psi_R$ is a formula over the empty language. Then a tuple $a$ satisfies $\phi_R$ in $G$ iff $e_N(a)$ satisfies $\phi_R$ in $G$ (because $e_N$ is an embedding) iff $e_N(a)$ satisfies $\psi_R$ in $G$ (as there are no edges on $e_N(a)$) iff $e_N(a)$ satisfies $\psi_R$ in the substructure induced by $e_N[V]$ (since $\psi_R$ does not contain any quantifiers). Thus, $\Gamma$ is isomorphic to the structure on $e_N[V]$ which has the relations defined by the formulas $\psi_R$; hence, $\Gamma$ is isomorphic to a structure with a first-order definition over the empty language. This structure has, of course, all injections as self-embeddings, and all permutations as automorphisms, and hence is model-complete; thus, the same is true for $\Gamma$.
\end{proof}

It follows from~\cite[Proposition 19]{tcsps-journal} that all reducts
of the linear order of the rationals $(\mathbb Q;<)$ are
model-complete as well. This is remarkable, since similar structures do not have this property -- for example, $(\mathbb Q;<,0)$ is first-order interdefinable with the structure $(\mathbb Q;<,[0,\infty))$ which is not model-complete.

We now turn to an even finer way of distinguishing reducts of an $\omega$-categorical structure, namely up to \emph{primitive positive interdefinability}. This is of importance in connection with the constraint satisfaction problem from the introduction, as we will describe in more detail in Section~\ref{sect:csp}. We call a formula \emph{primitive positive} iff it is existential positive and does not contain disjunctions. A \emph{clone} on domain $D$ is a set of finitary operations on $D$ which contains all projections (i.e., functions of the form $(x_1,\ldots,x_n)\mapsto x_i$) and which is closed under composition. A clone $\C$ is \emph{closed} (also called \emph{locally closed} or \emph{local} in the literature) iff for each $n\geq 1$, the set of $n$-ary functions in $\C$ is a closed subset of the space $D^{D^n}$, where $D$ is taken to be discrete. The closed clones on $D$ form a complete lattice with respect to inclusion -- the structure of this lattice has been studied in the universal algebra literature (see~\cite{GoldsternPinsker}, \cite{Pin-morelocal}). Similarly, the set of relational structures with domain $D$ which are \emph{primitive positive closed}, i.e., which contain all relations which they define by primitive positive formulas, forms a complete lattice. For a structure $\Gamma$, we define $\Pol(\Gamma)$ to consist of all finitary operations on the domain of $\Gamma$ which preserve all relations of $\Gamma$, i.e., an $n$-ary function $f$ is an element of $\Pol(\Gamma)$ iff for all relations $R$ of $\Gamma$ and all tuples $r_1,\ldots,r_n\in R$ the tuple $f(r_1,\ldots,r_n)$ is an element of $R$. It is easy to see that $\Pol(\Gamma)$ is always a closed clone. Observe also that $\Pol(\Gamma)$ is a generalization of $\End(\Gamma)$ to higher (finite) arities.

\begin{thm}[Bodirsky and Ne\v{s}et\v{r}il~\cite{BodirskyNesetrilJLC}]\label{thm:clones-pp}
    Let $\Delta$ be $\omega$-categorical. Then the mapping $\Gamma\mapsto \Pol(\Gamma)$ is an antiisomorphism between the lattice of primitive positive closed reducts of $\Delta$ and the lattice of closed clones containing $\Aut(\Delta)$. The inverse mapping is given by $\C\mapsto\Inv(\C)$.
\end{thm}

It turns out that even for the empty structure $(X;=)$, the lattice of primitive positive closed reducts is probably too complicated to be completely described -- the lattice has been thoroughly investigated in~\cite{BodChenPinsker}.

\begin{thm}[Bodirsky, Chen, Pinsker~\cite{BodChenPinsker}]
    The structure $(X;=)$ (where $X$ is countably infinite), and therefore all countably infinite structures, have $2^{\aleph_0}$ reducts up to primitive positive interdefinability.
\end{thm}

Fortunately, it is sometimes sufficient in applications to understand only parts of this lattice. We will see examples of this in Section~\ref{sect:csp}. 

\section{Ramsey Classes}\label{sect:ramseyclasses}

While Theorems~\ref{thm:groups-fo},~\ref{thm:monoids-expos} and~\ref{thm:clones-pp} provide a theoretical method for determining reducts of an $\omega$-categorical structure $\Delta$ by transforming them into sets of functions on $\Delta$, understanding these infinite objects could turn out difficult without further tools for handling them. We will now focus on structures which have the additional property that they are reducts of an \emph{ordered Ramsey structure} that is homogeneous in a finite relational signature; such structures are $\omega$-categorical since homogeneous structures in a finite language are $\omega$-categorical and since reducts of $\omega$-categorical structures are $\omega$-categorical. 
This is less restrictive than it might appear at first sight: 
we remark that it could be the case that all homogeneous structures
with a finite relational signature are reducts of ordered homogeneous Ramsey structures with
a finite relational signature (that is, we do not know of a counterexample). 
It turns out that in this context, certain infinite functions can be represented by finite ones, making classification projects more feasible.

\begin{defn}\label{defn:orderd}
    A structure is called \emph{ordered} iff it has a total order among its relations.
\end{defn}

\begin{definition}
    Let $\tau$ a relational signature. For $\tau$-structures $\S,\H,\P$ and an integer $k \geq 1$, we write $\S
    \rightarrow (\H)^\P_k$ iff for every $k$-coloring $\chi$ of the
    copies of $\P$ in $\S$ there exists a copy $\H'$ of $\H$ in $\S$
    such that all copies of $\P$ in $\H'$ have the same color under
    $\chi$.
\end{definition}

\begin{definition}
    A class $\C$ of finite $\tau$-structures which is closed under isomorphisms, induced substructures, and with the joint embedding property (see~\cite{HodgesLong}) is called a \emph{Ramsey class} iff it is closed under substructures and for all $k\geq 1$ and all $\H, \P\in \C$ there exists $\S$ in $\C$ such that $\S    \rightarrow (\H)^\P_k$.
\end{definition}

\begin{definition}
    A relational structure is called \emph{Ramsey} iff its \emph{age}, i.e., the set of finite structures isomorphic to a finite induced substructure, is a Ramsey class.
\end{definition}

Examples of Ramsey structures are the dense linear order $(\mathbb{Q};<)$ and the \emph{ordered random graph} $(V;E,<)$, i.e., 
the Fra\"{i}ss\'{e} limit of the class of finite ordered graphs. We remark that the random graph itself is not Ramsey, but since it is a reduct of the ordered random graph, the methods we are about to expose apply as well. 

We will now see that one can find regular patterns in the behavior of any function acting on an ordered Ramsey structure which is $\omega$-categorical.

\begin{definition}
    Let $\Gamma$ be a structure. The \emph{type} $\tp(a)$ of an $n$-tuple $a \in \Gamma$ is the set of first-order formulas with
     free variables $x_1,\dots,x_n$ that hold for $a$ in $\Gamma$.
\end{definition}

We recall the classical theorem of Ryll-Nardzewski about the number of types in $\omega$-categorical structures.

\begin{thm}[Ryll-Nardzewski]\label{thm:RN}
	The following are equivalent for a countable structure $\Gamma$ in a countable language.
\begin{itemize}
	\item $\Gamma$ is $\omega$-categorical, i.e., any countable model of the theory of $\Gamma$ is isomorphic to $\Gamma$.
	\item $\Gamma$ has for all $n\geq 1$ only finitely many different types of $n$-tuples. 
\end{itemize}
\end{thm}

We also mention that moreover, as a well-known consequence of the proof of this theorem, two tuples in a countable $\omega$-categorical structure have the same type if and only if there is an automorphism of $\Gamma$ which sends one tuple to the other.

\begin{definition}
    A \emph{type condition} between two structures $\Gamma_1,\Gamma_2$ is a pair $(t_1,t_2)$, where each $t_i$ is a type of an $n$-tuple in $\Gamma_i$. A function $f:\Gamma_1\To \Gamma_2$ \emph{satisfies} a type condition $(t_1,t_2)$ if for all $n$-tuples $(a_1,\ldots,a_n)$ of type $t_1$, the $n$-tuple $(f(a_1),\ldots,f(a_n))$ is of type $t_2$.

    A \emph{behavior} is a set of type conditions between two structures. A function \emph{has behavior $B$} if it satisfies all the type conditions of the behavior $B$. A behavior $B$ is called \emph{complete} iff for all types $t_1$ of tuples in $\Gamma_1$ there is a type $t_2$ of a tuple in $\Gamma_2$ such that $(t_1,t_2)\in B$.

    A function $f: \Gamma_1 \To \Gamma_2$ is \emph{canonical} iff it has a complete behavior. If $F\subseteq \Gamma_1$, then we say that $f$ is \emph{canonical on $F$} if its restriction to $F$ is canonical.
\end{definition}

Observe that the function $\leftrightarrow$ of Theorem~\ref{thm:cameron5} is canonical for the structure $(\mathbb{Q};<)$. The function $\circlearrowright$ is not, but it is canonical on each of the intervals $(-\infty,\alpha)$ and $(\alpha,\infty)$. For the random graph, the function $-$ of Theorem~\ref{thm:thomas5} is canonical, while $\sw$ is canonical on $V\setminus\{0\}$. Also, $\sw$ is canonical as a function from $(V;E,0)$ to $(V;E)$, where $(V;E,0)$ denotes the structure obtained from $(V;E)$ by adding a new constant symbol for the element $0$ by which we defined the function $\sw$. Moreover, the constant function and $e_E, e_N$ of Theorem~\ref{thm:randomMinimalMonoids} are canonical on $(V;E)$. We will now show that it is no coincidence that canonical functions are that ubiquitous.

\begin{defn}
    Let $\Delta$ be a structure. A property $P$ \emph{holds for arbitrarily large finite substructures of $\Delta$} iff for all finite substructures $F\subseteq \Delta$ there is a copy of $F$ in $\Delta$ for which $P$ holds.
\end{defn}

The following observation is just an easy application of the definition of a Ramsey class, but crucial in understanding functions on ordered Ramsey structures.

\begin{lem}\label{lem:canonicalOnArbitrarilyLargeFinite}
    Let $\Delta$ be ordered Ramsey and $\omega$-categorical, and let $f: \Delta\To \Delta$. Then $f$ is canonical on arbitrarily large finite substructures.
\end{lem}

The proof goes along the following lines: Let $F$ be any finite substructure of $\Delta$. Then the function $f$ induces a mapping from the tuples in $\Delta$ to the set of types in $\Delta$ (each tuple is sent to the type of its image under $f$). If we restrict this mapping to tuples of length at most the size of $F$, then since $\Delta$ is $\omega$-categorical, the range of this restriction is finite by Theorem~\ref{thm:RN}, and thus is a $k$-coloring of tuples for some finite $k$.  Now apply the Ramsey property once for every type of tuple that occurs in $F$ -- see~\cite{BodPinTsa} for details. We remark that this lemma would be false if one dropped the order assumption, which implies that coloring induced substructures and coloring tuples in $\Delta$ are one and the same thing.

The motivation for working with ordered Ramsey structures is the rough idea that all ``important'' functions can be assumed to be canonical. While this is simply false when stated boldly like this, it is still true for some functions when the idea is further refined, as we will show in the following. Observe that if $\Delta$ is $\omega$-categorical, then for each $n\geq 1$ there are only finitely many possible type conditions for $n$-types over $\Delta$ (Theorem~\ref{thm:RN}). Suppose that $\Delta$ has in addition a finite language and \emph{quantifier elimination}, 
i.e., every first-order
formula in the language of $\Delta$ is equivalent to a quantifier-free formula over $\Delta$; this follows in particular from homogeneity in a finite language. Then, if $n(\Delta)$ is the largest arity of its relations, then a function $f:\Delta\To\Delta$ is canonical iff for every type $t_1$ of an $n(\Delta)$-tuple in $\Delta$ there is a type $t_2$ in $\Delta$ such that $f$ satisfies the type condition $(t_1,t_2)$. In other words, the complete behavior of $f$ is already determined by its behavior on $n(\Delta)$-types. Hence, a canonical function on $\Delta$ is essentially a function on the $n(\Delta)$-types of $\Delta$ -- a finite object.

\begin{defn}
    Let $f,g: \Delta\To \Delta$. We say that $f$ \emph{generates} $g$ over $\Delta$ iff $g$ is contained in the smallest closed monoid containing $f$ and $\Aut(\Delta)$. Equivalently, for every finite subset $F$ of $\Delta$, there exists a term $\beta\circ f\circ\alpha_1\circ f\circ\alpha_2\circ\cdots\circ f\circ \alpha_n$, where $\beta,\alpha_i\in\Aut(\Delta)$, which agrees with $g$ on $F$.
\end{defn}

\begin{prop}\label{lem:generatesCanonical}
    Let $\Delta$ be a structure in a finite language which is ordered, Ramsey, and homogeneous. Let $f: \Delta\To \Delta$. Then $f$ generates a canonical function $g:\Delta\To\Delta$.
\end{prop}
\begin{proof}[First proof]
    Let $(F_i)_{i\in\omega}$ be an increasing sequence of finite substructures of $\Delta$ such that $\bigcup_{i\in\omega} F_i=\Delta$. By Lemma~\ref{lem:canonicalOnArbitrarilyLargeFinite}, for each $i\in\omega$ we find a copy $F_i'$ of $F_i$ in $\Delta$ on which $f$ is canonical. Since there are only finitely many possibilities of canonical behavior, one behavior occurs an infinite number of times; thus, by thinning out the sequence, we may assume that the behavior is the same on all $F_i'$. By the homogeneity of $\Delta$, there exist automorphisms $\alpha_i$ of $\Delta$ sending $F_i$ to $F_i'$, for all $i\in\omega$. Also, since the behavior on all the $F_i'$ is the same, we can inductively pick automorphisms $\beta_i$ of $\Delta$ such that $\beta_{i+1}\circ f\circ\alpha_{i+1}$ agrees with $\beta_i\circ f\circ\alpha_i$ on $F_i$, for all $i\in\omega$. The union over the functions $\beta_i\circ f\circ \alpha_i: F_i\To\Delta$ is a canonical function on $\Delta$.
\end{proof}
\begin{proof}[Second proof]
    The identity function $\id: \Delta\To\Delta$ is generated by $f$ and is canonical.
\end{proof}

The problem with the preceding lemma is the second proof, which makes it trivial. What we really want is that $f$ generates a canonical function $g$ which represents $f$ in a certain sense -- it should be possible to retain specific properties of $f$ when passing to the canonical functions. For example, we could wish that if $f$ violates a certain relation, then so does $g$; or, if $f$ is not an automorphism of $\Delta$, we will look for a canonical function $g$ which is not an automorphism of $\Delta$ either.

We are now going to refine our method, and fix constants $c_1,\dots,c_n$
such that $f \notin \Aut(\Delta)$ is witnessed on $\{c_1,\dots,c_n\}$. 
We then consider $f$ as a function from $(\Delta,c_1,\dots,c_n)$ to $\Delta$, 
where $(\Delta,c_1,\dots,c_n)$ denotes the expansion of $\Delta$ by the constants
$c_1,\dots,c_n$. It turns out that $f$ is canonical on arbitrarily large
substructures of $(\Delta,c_1,\dots,c_n)$, and that it generates a canonical
function $g:  (\Delta,c_1,\dots,c_n) \rightarrow \Delta$ which agrees with $f$ on
$c_1,\dots,c_n$; in particular, $g$ is not an automorphism of $\Delta$, 
and the problem of triviality in Proposition 20 no longer occurs. 
In order to do this, we must assure that  $(\Delta,c_1,\ldots,c_n)$ still has the Ramsey property. This leads us into topological dynamics. 

\section{Topological Dynamics}\label{sect:topologicaldynamics}
We have seen in the previous section that our approach crucially relies on the fact that
when an ordered homogeneous Ramsey structure is expanded by finitely many constants,
the expansion is again Ramsey (it is clear that the expansion is again ordered and homogeneous). To prove this, we use a characterization in topological dynamics of those ordered homogeneous structures
which are Ramsey.

Recall that a \emph{topological group} is an (abstract) group $G$ together with a topology on the elements of $G$ such that 
$(x,y) \mapsto xy^{-1}$ is continuous from $G^2$ to $G$. In other words, 
we require that the binary group operation and the inverse function are continuous.  


\begin{defn}
    A topological group is \emph{extremely amenable} iff any continuous action of the group on a compact Hausdorff space has a fixed point.
\end{defn}

Kechris, Pestov and Todorcevic have characterized the Ramsey property of the age of an ordered homogeneous structure by means of extreme amenability in the following theorem.

\begin{thm}[Kechris, Pestov, Todorcevic \cite{Topo-Dynamics}]\label{thm:KPT}
    Let $\Delta$ be an ordered homogeneous relational structure. Then the age of $\Delta$ has the Ramsey property iff $\Aut(\Delta)$ is extremely amenable.
\end{thm}

This theorem can be applied to provide a short and elegant proof of the following.

\begin{prop}[Bodirsky, Pinsker and Tsankov~\cite{BodPinTsa}]\label{prop:addingConstantsPreservesRamsey}
    Let $\Delta$ be ordered, Ramsey, and homogeneous, and let $c_1,\ldots,c_n\in \Delta$. Then $(\Delta,c_1,\ldots,c_n)$ is Ramsey as well.
\end{prop}

When $\Delta$ is ordered, Ramsey, and homogeneous, then $\Aut(\Delta)$ is extremely amenable. Note that the automorphism group of $(\Delta,c_1,\ldots,c_n)$ is an open subgroup of $\Aut(\Delta)$. The proposition thus follows directly from the following fact -- confer~\cite{BodPinTsa}.

\begin{lem}
    Let $G$ be an extremely amenable group, and let $H$ be an open subgroup of $G$. Then $H$ is extremely amenable.
\end{lem}

\section{Minimal Functions}\label{sect:minimalfunctions}

The results of the preceding section provide a tool for ``climbing up'' the lattice of closed monoids containing the automorphism group of an ordered Ramsey structure which is homogeneous and has a finite language.

\begin{defn}\label{defn:minimalClone}
    Let $\C, \D$ be closed clones. Then $\D$ is called \emph{minimal above $\C$} iff $\D\supseteq\C$ and there are no closed clones between $\C$ and $\D$.
\end{defn}

Observe that transformation monoids can be identified with those clones which have the property that all their functions depend on only one variable. Hence, Definition~\ref{defn:minimalClone} also provides us with a notion of a minimal closed monoid above another closed monoid.

It follows from Theorem~\ref{thm:clones-pp} and Zorn's Lemma that if $\Delta$ is an $\omega$-categorical structure in a finite language, then every closed clone containing $\Pol(\Delta)$ contains a minimal closed clone above $\Pol(\Delta)$. Similarly, as a consequence of Theorem~\ref{thm:monoids-expos}, every closed monoid containing $\End(\Delta)$ contains a minimal closed monoid.

For closed permutation groups, minimality can be defined analogously. Then Theorem~\ref{thm:groups-fo} implies that for $\omega$-categorical structures $\Delta$ in a finite language, every closed permutation group containing $\Aut(\Delta)$ contains a minimal closed permutation group above $\Aut(\Delta)$.

Clearly, if a closed clone $\D$ is minimal above $\C$, then any function $f\in\D\setminus\C$ generates $\D$ with $\C$ (i.e., $\D$ is the smallest closed clone containing $f$ and $\C$) -- similar statements hold for monoids and groups. In the case of clones and monoids and in the setting of reducts of ordered Ramsey structures which are homogeneous in a finite language, we can standardize such generating functions. This is the contents of the coming subsections.

\subsection{Minimal unary functions}\label{subsect:minimalUnary}

Adapting the proof of Lemma~\ref{lem:generatesCanonical}, with the use of the Proposition~\ref{prop:addingConstantsPreservesRamsey}, one can show the following.

\begin{lem}\label{lem:generatesCanonicalWithConstants}
    Let $\Delta$ be ordered, Ramsey, homogeneous, and of finite language. Let $f: \Delta\To \Delta$, and let $c_1,\ldots,c_n\in \Delta$. Then $f$ together with $\Aut(\Delta)$ generates a function which agrees with $f$ on $\{c_1,\ldots,c_n\}$ and which is canonical as a function from $(\Delta,c_1,\ldots,c_n)$ to $\Delta$.
\end{lem}

Let $\Gamma$ be a finite language reduct of a structure $\Delta$ which is ordered, Ramsey, homogeneous, and of finite language, and let $\N$ be a minimal closed monoid containing $\End(\Gamma)$. Then, setting $n(\Gamma)$ to be the largest arity of the relations of $\Gamma$, we can pick constants $c_1,\ldots,c_{n(\Gamma)}\in \Gamma$ and a function $f\in \N\setminus\End(\Gamma)$ such that $f\nin\End(\Gamma)$ is witnessed on $\{c_1,\ldots,c_{n(\Gamma)}\}$. By the preceding lemma, $f$ and $\Aut(\Delta)$ generate a function $g$ which behaves like $f$ on $\{c_1,\ldots,c_{n(\Gamma)}\}$ and which is canonical as a function from $(\Delta,c_1,\ldots,c_{n(\Gamma)})$ to $\Delta$. This function $g$, together with $\End(\Gamma)$, generates $\N$. Since there are only finitely many choices for the type of the tuple $(c_1,\ldots,c_{n(\Gamma)})$ and for each choice only finitely many behaviors of functions from $(\Delta,c_1,\ldots,c_{n(\Gamma)})$ to $\Delta$, we get the following.

\begin{prop}[Bodirsky, Pinsker, Tsankov~\cite{BodPinTsa}]\label{prop:finiteMinimalReducts}
    Let $\Gamma$ be a finite language reduct of a structure $\Delta$ which is ordered, Ramsey, homogeneous, and of finite language. Then the number of minimal closed monoids above $\End(\Gamma)$ is finite, and each such monoid is generated by $\End(\Gamma)$ plus a canonical function $g:(\Delta,c_1,\ldots,c_{n(\Gamma)})\To \Delta$, for constants $c_1,\ldots,c_{n(\Gamma)}\in\Gamma$.
\end{prop}

Since for every relation $R$ of $\Gamma$ we can add its negation to the language, we get the following

\begin{cor}\label{cor:finiteMinimalSelfEmbeddings}
    Let $\M$ be the monoid of self-embeddings of a finite-language structure $\Gamma$ which is a reduct of a structure $\Delta$ which is ordered, Ramsey, homogeneous, and of finite language. Then the number of minimal closed monoids above $\M$ is finite, and each such monoid is generated by $\M$ and a canonical function $g:(\Delta,c_1,\ldots,c_{n(\Gamma)})\To \Delta$.
\end{cor}

The following is an example for the random graph $G=(V;E)$. Since $G$ is model-complete, its monoid of self-embeddings is just the topological closure $\cl{\Aut(G)}$ of $\Aut(G)$ in the space $V^V$. Therefore, the minimal closed monoids above the monoid of self-embeddings of $G$ are just the minimal closed monoids above $\cl{\Aut(G)}$.

\begin{thm}[Thomas~\cite{Thomas96}]\label{thm:randomMinimalUnary}
    Let $G=(V;E)$ be the random graph. The minimal closed monoids containing $\cl{\Aut(G)}$ are the following:
    \begin{itemize}
        \item The monoid generated by a constant operation with $\Aut(G)$.
        \item The monoid generated by $e_E$ with $\Aut(G)$.
        \item The monoid generated by $e_N$ with $\Aut(G)$.
        \item The monoid generated by $-$ with $\Aut(G)$.
        \item The monoid generated by $\sw$ with $\Aut(G)$.
    \end{itemize}
\end{thm}

\subsection{Minimal higher arity functions}\label{subsect:MinimalHigherArity}

We now generalize the concepts from unary functions and monoids to higher arity functions and clones.

\begin{defn}\label{defn:typesOn Products}
    Let $\Delta$ be a structure. For $1\leq i\leq m$ and a tuple $x$ in the power $\Delta^m$, we write
    $x_i$ for the $i$-th coordinate of $x$. The \emph{type} of a sequence of tuples $a^1,\ldots,a^n\in \Delta^m$, denoted by $\tp(a^1,\ldots,a^n)$, is the cartesian product of the types
    of $(a^1_i,\ldots,a^n_i)$ in $\Delta$.
\end{defn}

With this definition, the notions of \emph{type condition}, \emph{behavior}, \emph{complete behavior}, and \emph{canonical} generalize in complete analogy from functions $f:\Gamma_1\To\Gamma_2$ to functions $f:\Gamma_1^m\To\Gamma_2$, for structures $\Gamma_1, \Gamma_2$. It can be shown that for ordered structures, the Ramsey property is not lost when going to products; an example of a proof can be found in~\cite{BodPinTsa}.

\begin{prop}
\label{prop:ORPL}
    Let $\Delta$ be ordered and Ramsey, and let $m\geq 1$. Let moreover a number $k\geq 1$, an $n$-tuple $(a^1,\ldots,a^n) \in\Delta^m$, and finite $F_i\subseteq\Delta$ be given for $1\leq i\leq m$. Then  there exist finite $S_i\subseteq\Delta$ with the property that whenever the $n$-tuples in $S_1\mult\cdots\mult S_m$ of type $\tp(a^1,\ldots,a^n)$ are colored with $k$ colors, then there are copies $F_i'$ of $F_i$ in $S_i$ such that the coloring is constant on $F_1'\mult\cdots\mult F_m'$.
\end{prop}

We remark that Proposition~\ref{prop:ORPL} does not hold in general if $\Delta$ is not assumed to be ordered -- an example for the random graph can be found in~\cite{RandomMinOps}. Similarly to the unary case (Proposition~\ref{prop:finiteMinimalReducts}), one gets the following.

\begin{prop}[Bodirsky, Pinsker, Tsankov~\cite{BodPinTsa}]\label{prop:canonicalMinimalClones}
    Let $\Gamma$ be a finite language reduct of a structure $\Delta$ which is ordered, Ramsey, homogeneous and of finite language. Then every minimal closed clone above $\Pol(\Gamma)$ is generated by $\Pol(\Gamma)$ and a canonical function $g:(\Delta,c_1,\ldots,c_{k})^m\To \Delta$, where $m\geq 1$, $k\geq 0$, and $c_1,\ldots,c_{k}\in \Delta$. Moreover, $m$ only depends on the number of $n(\Gamma)$-types in $\Gamma$ (and not on the clone), and $k$ only depends on $m$ and $n(\Gamma)$, and the number of minimal closed clones above $\Pol(\Gamma)$ is finite.
\end{prop}

In the case of minimal closed clones above an endomorphism monoid, the arity of the generating canonical functions can be further reduced as follows.

\begin{prop}[Bodirsky, Pinsker, Tsankov~\cite{BodPinTsa}]\label{prop:canonicalMinimalClonesAboveEnd}
    Let $\Gamma$ be a finite language reduct of a structure $\Delta$ which is ordered, Ramsey, homogeneous and of finite language. Then every minimal closed clone above $\End(\Gamma)$ is generated by $\End(\Gamma)$ and a canonical function $g:(\Delta,c_1,\ldots,c_{n(\Gamma)})\To \Delta$, or by $\End(\Gamma)$ and a canonical function $g:(\Delta,c_1,\ldots,c_m)^m\To\Delta$, where $m$ only depends on the number of $2$-types in $\Gamma$ (and not on the clone). In particular, the number of minimal closed clones above $\End(\Gamma)$ is finite.
   \end{prop}

Using this technique, the minimal closed clones containing the automorphism group of the random graph $G=(V;E)$ have been determined. In the following, let $f: V^2 \rightarrow V$ be a binary
operation; we now define some possible behaviors for $f$.
We say that $f$ is
\begin{itemize}
\item \emph{of type $p_1$} iff for all $x_1,x_2,y_1,y_2 \in V$ with $x_1 \neq x_2$
and $y_1 \neq y_2$ we have
$E(f(x_1,y_1),f(x_2,y_2))$ if and only if
$E(x_1,x_2)$;
\item \emph{of type $\max$} iff for all $x_1,x_2,y_1,y_2 \in V$ with $x_1 \neq x_2$ and $y_1 \neq y_2$ we have
$E(f(x_1,y_1),f(x_2,y_2))$ if and only if
$E(x_1,x_2)$ or $E(y_1,y_2)$;
\item \emph{balanced in the first argument} iff for all $x_1,x_2,y \in V$
with $x_1 \neq x_2$ we have $E(f(x_1,y),f(x_2,y))$ if and only if $E(x_1,x_2)$;
\item \emph{balanced in the second argument}
iff $(x,y) \mapsto f(y,x)$ is balanced in the first argument;
\item \emph{$E$-dominated in the first argument}
iff for all $x_1,x_2,y \in V$ with $x_1 \neq x_2$ we have that $E(f(x_1,y),f(x_2,y))$;
\item \emph{$E$-dominated in the second argument} iff
$(x,y) \mapsto f(y,x)$ is $E$-dominated in the first argument.
\end{itemize}

The \emph{dual} of an operation $f(x_1,\ldots,x_n)$ on $V$ is defined by $-f(-x_1,\ldots,-x_n)$.

\begin{theorem}[Bodirsky and Pinsker~\cite{RandomMinOps}]\label{thm:minimalRandomClones}
Let $G=(V;E)$ be the random graph, and let $\C$ be a minimal closed clone above $\cl{\Aut(G)}$. Then $\C$ is generated by $\Aut(G)$ together with one of the unary functions of Theorem~\ref{thm:randomMinimalUnary}, or by $\Aut(G)$ and one of the following canonical operations from $G^2$ to $G$:
\begin{itemize}
\item a binary injection of type $p_1$ that is balanced in both arguments;
\item a binary injection of type $\max$ that is balanced in both arguments;
\item a binary injection of type $\max$ that is $E$-dominated in both arguments;
\item a binary injection of type $p_1$ that is $E$-dominated in both arguments;
\item a binary injection of type $p_1$ that is balanced in the first
and $E$-dominated in the second argument;
\item the dual of one of the last four operations.
\end{itemize}
\end{theorem}

In~\cite{BodPin-Schaefer}, the technique of canonical functions was applied again to climb up further in the lattice of closed clones above $\Aut(G)$ -- we will come back to this in Section~\ref{sect:csp}.

Another example are the minimal closed clones containing all permutations of a countably infinite base set $X$. Observe that the set $\S_X$ of all permutations on $X$ is the automorphism group of the structure $(X;=)$ which has no relations.

\begin{theorem}[Bodirsky and K\'{a}ra~\cite{ecsps}; cf.~also~\cite{BodChenPinsker}]\label{thm:minimalAboveS}
The minimal closed clones containing $\cl{\S_X}$ on a countably infinite set $X$ are:
\begin{itemize}
\item The closed clone generated by $\S_X$ and any constant operation;
\item The closed clone generated by $\S_X$ and any binary injection.
\end{itemize}
\end{theorem}

Observe that any constant operation and any binary injection on $X$ are canonical operations for the structure $(X;=)$.

We end this section with a last example which lists the minimal closed clones containing the self-embdeddings of the dense linear order $(\mathbb{Q};<)$. As with the random graph and the empty structure, since $(\mathbb{Q};<)$ is model-complete it follows that the monoid of self-embeddings of $(\mathbb{Q};<)$ is just the closure of $\Aut((\mathbb{Q};<))$ in $\mathbb{Q}^\mathbb{Q}$.

Let $\lex$ be a binary operation on $\mathbb{Q}$ such that $\lex(a,b) < \lex(a',b')$ iff either $a < a'$ or $a = a'$
and $b < b'$, for all $a,a',b,b'\in \mathbb{Q}$. Observe that $\lex$ is canonical as a function from $\mathbb{Q}^2$ to $\mathbb{Q}$. Next, let $\pp$ be an arbitrary binary operation on $\mathbb{Q}$ such that for all $a,a',b,b'\in \mathbb{Q}$ we have $\pp(a, b)\leq \pp(a', b')$ iff one of the following cases applies:
\begin{itemize}
    \item $a \leq 0$ and $a \leq a'$;
    \item $0 < a$, $0 <a'$, and $b \leq b'$.
\end{itemize}
The name of the operation $\pp$ stands for ``projection-projection'', since
the operation behaves as a projection to the first argument for negative first argument, and a projection to the second argument for positive first argument. Observe that $\pp$ is canonical if we add the origin as a constant to the language. Finally, define the \emph{dual} of an operation $f(x_1,\ldots,x_n)$ on $\mathbb{Q}$ by ${\leftrightarrow}(f({\leftrightarrow}(x_1),\ldots,{\leftrightarrow}(x_n)))$.

\begin{theorem}[Bodirsky and K\'{a}ra~\cite{tcsps-journal}]\label{thm:minimalDLOClones}
Let $(\mathbb{Q};<)$ be the order of the rationals, and let $\C$ be a minimal closed clone above $\cl{\Aut((\mathbb{Q};<))}$. Then $\C$ is generated by $\Aut((\mathbb{Q};<))$ together with one of the following operations:
\begin{itemize}
\item a constant operation;
\item the operation $\leftrightarrow$;
\item the operation $\circlearrowright$;
\item the operation $\lex$;
\item the operation $\pp$;
\item the dual of $\pp$.
\end{itemize}
\end{theorem}

\section{Decidability of Definability}\label{sect:decidability}

We turn to another application of the ideas of the last sections. Consider the following computational problem for a structure $\Gamma$: Input are quantifier-free formulas $\phi_0,\ldots,\phi_n$ in the language of $\Gamma$ defining relations $R_0,\ldots,R_n$ on the domain of $\Gamma$, and the question is whether $R_0$ can be \emph{defined} from $R_1,\ldots,R_n$. As in Section~\ref{sect:reducts}, ``defined'' can stand for ``first-order defined'' or syntactic restrictions of this notion. We denote this computational problem by $\Exprep(\Gamma)$ and $\Exprpp(\Gamma)$ if we consider existential positive and primitive positive definability, respectively.

For \emph{finite} structures $\Gamma$ the problem $\Exprpp(\Gamma)$ is in co-NEXPTIME (and in particular decidable), and has recently shown to be co-NEXPTIME-hard~\cite{Willard-cp10}. For infinite  structures $\Gamma$, the decidability of $\Exprpp(\Gamma)$ is not obvious. An algorithm for primitive positive definability has theoretical and practical consequences in the study of the computational complexity of CPSs (which we will consider in Section~\ref{sect:csp}).
It is motivated by the fundamental fact that expansions of structures $\Gamma$ by primitive positive relations do not change the complexity of $\Csp(\Gamma)$. On a practical side, it turns out that hardness of a CSP can usually be shown by presenting primitive positive definitions of relations for which it is known that the CSP is hard. Therefore, a procedure that decides primitive positive definability of a given relation might be a useful tool to determine the computational complexity of CSPs.

Using the methods of the last sections, one can show decidability of $\Exprep(\Gamma)$ and $\Exprpp(\Gamma)$ for certain infinite structures $\Gamma$. The following uses the same terminology as in~\cite{MacphersonSurvey}.

\begin{definition}
    We say that a class $\mathcal C$ of finite $\tau$-structures (or a $\tau$-structure with age $\mathcal C$) is \emph{finitely bounded}
    if there exists a finite set of finite $\tau$-structures $\mathcal F$ such for all finite $\tau$-structures $A$ we have 
    that $A \in \mathcal C$ iff no structure from $\mathcal F$
    embeds into $A$.
\end{definition}

\begin{thm}[Bodirsky, Pinsker, Tsankov~\cite{BodPinTsa}]\label{thm:decidability}
Let $\Delta$ be ordered, Ramsey, homogeneous, and of finite language, and let $\Gamma$ be a finite language reduct of $\Delta$.
Then $\Exprep(\Gamma)$ and $\Exprpp(\Gamma)$ are decidable.
\end{thm}

Examples of structures $\Delta$ that satisfy the assumptions
of Theorem~\ref{thm:decidability} are ${\dlo}$,
the Fra\"{i}ss\'{e} limit of ordered finite graphs (or tournaments~\cite{RamseyClasses}), the Fra\"{i}ss\'{e} limit of
finite partial orders with a linear extension~\cite{RamseyClasses}, the homogeneous universal `naturally ordered' $C$-relation~\cite{BodirskyPiguet}, just to name a few. CSPs for structures that are definable in such structures are abundant
in particular for qualitative reasoning calculi in Artificial Intelligence.

We want to point out that that decidability of primitive positive definability is already non-trivial when $\Gamma$ is trivial from a model-theoretic perspective: for the case that $\Gamma$ is the structure $(X; =)$ (where $X$ is countably infinite), the decidability of $\Exprpp(\Gamma)$ has been posed as an open problem in~\cite{BodChenPinsker}. Theorem~\ref{thm:decidability} solves this problem,
since $(X;=)$ is isomorphic to a reduct of the structure $({\mathbb Q}; <)$, which is clearly finitely bounded, homogeneous, 
ordered, and Ramsey.

The proof of Theorem~\ref{thm:decidability} goes along the following lines, and is based on the results of the last sections. We outline the algorithm for $\Exprpp(\Gamma)$; the proof for $\Exprep(\Gamma)$ is a subset. So the input are formulas $\phi_0,\ldots,\phi_n$ defining relations $R_0,\ldots,R_n$, and we have to decide whether $R_0$ has a primitive positive definition from $R_1,\ldots, R_n$. Let $\Theta$ be the structure which has $R_1,\ldots,R_n$ as its relations. By Theorem~\ref{thm:clones-pp}, $R_0$ is not primitive positive definable from $R_1,\ldots, R_n$ if and only if there is a finitary function $f\in\Pol(\Theta)$ which violates $R_0$. By the ideas of the last section, such a polymorphism can be chosen to be canonical as a function from $(\Delta,c_1,\ldots,c_k)^m$ to $\Delta$, where $c_i\in\Delta$. Such canonical functions are essentially finite objects since they can be represented as functions on types. Therefore, the algorithm can then check for a given canonical function whether it is a polymorphism of $\Theta$ and whether it violates $R_0$. Also, $k$ and $m$ can be calculated from the input, and so there are only finitely many complete behaviors to be checked. Finally, the additional assumption that $\Delta$ be finitely bounded allows the algorithm to check whether a function on types really comes from a function on $\Delta$. We refer to~\cite{BodPinTsa} for details.

\section{Interpretability}
\label{sect:interpret}
Many $\omega$-categorical structures can be derived from
other $\omega$-categorical structures via
first-order interpretations. In this section we will discuss the fact
already mentioned in the introduction
that bi-interpretations can be used to transfer the Ramsey
property from one structure to another.
A special type of
interpretations, called \emph{primitive positive interpretations},
will become important in Section~\ref{sect:csp}.
The definition of interpretability we use is standard, and
follows~\cite{HodgesLong}.


When $\Delta$ is a structure with signature $\tau$,
and $\delta(x_1,\dots,x_k)$ is a first-order $\tau$-formula with the $k$ free variables $x_1,\dots,x_k$, we write $\delta(\Delta^k)$ for the
$k$-ary relation that is defined by $\delta$ over $\Delta$.

\begin{definition}
A relational $\sigma$-structure $\Gamma$ has a \emph{(first-order) interpretation} in a $\tau$-structure $\Delta$ if there exists a natural number $d$, called the \emph{dimension} of the interpretation, and
\begin{itemize}
\item a $\tau$-formula $\delta(x_1, \dots, x_d)$ -- called \emph{domain formula},
\item for each $k$-ary relation symbol $R$ in $\sigma$ a $\tau$-formula $\phi_R(\overline x_1, \dots, \overline x_k)$ where the $\overline x_i$ denote disjoint $d$-tuples of distinct variables -- called the \emph{defining formulas},
\item a $\tau$-formula $\phi_=(x_1,\dots,x_d,y_1,\dots,y_d)$, and
\item a surjective map $h: \delta(\Delta^d) \rightarrow \Gamma$ -- called \emph{coordinate map},
\end{itemize}
such that for all relations $R$ in $\Gamma$ and all tuples $\overline a_i \in \delta(\Delta^d)$
\begin{align*}
(h(\overline a_1), \dots, h(\overline a_k)) \in R \;
& \Leftrightarrow \;
\Delta \models \phi_R(\overline a_1, \dots, \overline a_k) \; , \text{ and } \\
h(\overline a_1)=h(\overline a_2)
& \Leftrightarrow \;  \Delta \models \phi_=(\overline a_1,\overline a_2) \; .
\end{align*}
\end{definition}
If the formulas $\delta$, $\phi_R$, and $\phi_=$ are all primitive positive, we say that $\Gamma$
has a \emph{primitive positive interpretation} in $\Delta$; many primitive positive interpretations can be found in Section~\ref{sect:csp}.
We say that $\Gamma$ is \emph{interpretable in} $\Delta$ \emph{with finitely many parameters} if there are $c_1,\dots,c_n \in \Delta$
such that $\Gamma$ is interpretable in the expansion of $\Delta$ by the singleton relations $\{c_i\}$ for all $1 \leq i \leq n$.
First-order \emph{definitions} are a special case of interpretations:
a structure $\Gamma$ is \emph{(first-order) definable} in $\Delta$
if $\Gamma$ has an interpretation in $\Delta$ of dimension one where
the domain formula is logically equivalent to true.

\begin{lemma}[see e.g. Theorem 7.3.8 in~\cite{HodgesLong}]
\label{lem:interpret}
If $\Delta$ is an $\omega$-categorical structure, then every
structure $\Gamma$ that is first-order interpretable in $\Delta$
with finitely many parameters is $\omega$-categorical as well.
\end{lemma}

The following nicely describes interpretability between structures in terms of the (topological)
automorphism groups of the structures.

\begin{theorem}[Ahlbrandt and Ziegler~\cite{AhlbrandtZiegler}; also see Theorem 5.3.5 and 7.3.7 in~\cite{HodgesLong}]
  Let $\Delta$ be an $\omega$-categorical structure with at least two
  elements. Then a structure
  $\Gamma$ has a first-order interpretation in $\Delta$
  if and only if there is a continuous
  group homomorphism $f: \Aut(\Delta) \to \Aut(\Gamma)$ such that
  the image of $f$ has finitely many orbits in its action on $\Gamma$.
\end{theorem}

Note that if $\Gamma_2$ has a $d$-dimensional interpretation $I$ in $\Gamma_1$, and $\Gamma_3$
has an $e$-dimensional interpretation $J$ in $\Gamma_2$,
then $\Gamma_3$ has a natural $ed$-dimensional interpretation in
$\Gamma_1$, which we denote by $J \circ I$.
To formally
describe $J \circ I$, suppose that the signature of $\Gamma_i$
is $\tau_i$ for $i = 1,2,3$, and that
$I = (d,\delta,(\phi_R)_{R \in \tau_2}, \phi_=,h)$
where $d$ is the dimension, $\delta$ the domain formula,
$\phi_=$ and $(\phi_R)_{R \in \tau_2}$ the interpreting relations,
and $h$ the coordinate map.
Similarly,
let $J = (e,\gamma,(\psi_R)_{R \in \tau_3}, \psi_=,g)$.
We use the following.

\begin{lemma}[Theorem 5.3.2 in~\cite{HodgesLong}]
Let $\Gamma_1,\Gamma_2, I$ as in the preceding paragraph.
Then
for every first-order $\tau_2$-formula $\phi(x_1,\dots,x_k)$ there is
$\tau_1$-formula $$\phi^I(x^1_1,\dots,x^d_1,\dots,x^1_k,\dots,x_k^d)$$ such that for all $a_1,\dots,a_k \in \delta((\Gamma_1)^d)$
$$ \Gamma_2 \models \phi(h(a_1),\dots,h(a_k)) \;
\Leftrightarrow \; \Gamma_1 \models \phi^I(a_1,\dots,a_k) \; .$$
\end{lemma}

We can now define the interpretation $J \circ I$ as follows:
the domain formula $\eta$ is $\gamma^I$, and the defining formula for $R \in \tau_3$
is $(\psi_R)^I$.  The coordinate map is from $\eta((\Gamma_1)^{ed}) \rightarrow \Gamma_3$,
and defined by
$$(a^1_1,\dots,a^d_1,\dots,a^1_e,\dots,a^d_e) \; \mapsto \; g(h(a^1_1,\dots,a^d_1),\dots,h(a^1_e,\dots,a^d_e)) \; .$$

Two interpretations of $\Gamma$ in $\Delta$
with coordinate maps $h_1$ and $h_2$ are called \emph{homotopic}
if the relation $\{(\bar x,\bar y) \; | \; h_1(\bar x) = h_2(\bar y) \}$
is definable in $\Delta$.
The \emph{identity interpretation} of a structure $\Gamma$
is the 1-dimensional interpretation of $\Gamma$ in $\Gamma$ whose coordinate map is the identity.
Two structures $\Gamma$ and $\Delta$ are called \emph{bi-interpretable} if there is an interpretation $I$ of $\Gamma$ in $\Delta$ and an interpretation $J$ of $\Delta$ in $\Gamma$
such that both $I \circ J$ and $J \circ I$ are homotopic
to the identity interpretation (of $\Gamma$ and of $\Delta$, respectively).

\begin{theorem}[Ahlbrandt and Ziegler~\cite{AhlbrandtZiegler}]\label{thm:bi-interpret}
Two $\omega$-categorical structures $\Gamma$ and $\Delta$ are bi-interpretable if and only if $\Aut(\Gamma)$ and $\Aut(\Delta)$ are isomorphic as topological groups.
\end{theorem}

As a consequence of this result and Theorem~\ref{thm:KPT}
we obtain the following.

\begin{corollary}\label{cor:interpret-ramsey}
For ordered bi-interpretable
$\omega$-categorical homogeneous structures $\Gamma$ and $\Delta$,
one has the Ramsey property if and only if the other one has the Ramsey property.
\end{corollary}

We give an example. This corollary can be used to
deduce that
an important structure studied in temporal reasoning in artificial intelligence has the Ramsey property. For the relevance
of this fact in constraint satisfaction, see Section~\ref{sect:csp}.

We have already mentioned that the age of $({\mathbb Q}; <)$
has the Ramsey property. Let $\Gamma$ be the structure
whose elements are pairs $(x,y) \in {\mathbb Q}^2$ with $x<y$,
representing \emph{intervals},
and which contains all binary relations $R$ over those intervals
such that the relation $\{(x,y,u,v) \; | \; ((x,y),(u,v)) \in R\}$ is first-order
definable in $({\mathbb Q}; <)$.
Hence, $\Gamma$ has a 2-dimensional interpretation $I$ in $({\mathbb Q}; <)$, whose coordinate map $h_1$ is the identity map on
$D := \{(x,y) \in {\mathbb Q}^2 \; | \; x<y\}$.

The structure $\Gamma$ is known
under the name \emph{Allen's Interval Algebra} in artificial intelligence.
We claim that its age has the Ramsey property. Using the homogeneity of $\dlo$, it is easy to show that $\Gamma$ is
 homogeneous as well. 
By Corollary~\ref{cor:interpret-ramsey}, it suffices to show that
$\Gamma$ and $({\mathbb Q}; <)$ are bi-interpretable.
We first show that $({\mathbb Q};<)$ has an interpretation $J$
in $\Gamma$. The coordinate map $h_2$ of $J$ maps $(x,y) \in D$ to $x$.
The formula $\phi_=(a,b)$ is $R_0(a,b)$ where $R_0$ is the binary relation $\{((x,y),(u,v)) \; | \; x=u\}$ from $\Gamma$.
The formula $\phi_<(a,b)$ is $R_1(a,b)$ where $R_1$ is the binary relation
$\{((x,y),(u,v)) \; | \; x<u\}$.

We prove that $J \circ I$ is homotopic to the identity interpretation of $({\mathbb Q};<)$ in $({\mathbb Q};<)$.
This holds since the relation $\{(x,y,u) \in {\mathbb Q}^3  \; | \; h_2(h_1(x,y))=u\}$ has the first-order
definition $x=u$ in $({\mathbb Q}; <)$.
To show that $I \circ J$ is homotopic to the identity interpretation,
observe that the relation
$\{(a,b,c) \in D^3 \; | $
$h_1(h_2(a),h_2(b))=c\}$ has the first-order
definition $R_0(a,c) \wedge R_3(b,c)$ in $\Gamma$, where $R_0$ is the binary relation from $\Gamma$ as defined above, and $R_3$ is the binary relation $\{((x,y),(u,v))\in\Gamma^2 \; | \; x=v \}$ from $\Gamma$. This  shows that $\Gamma$ and $({\mathbb Q}; <)$ are bi-interpretable.

\section{Complexity of Constraint Satisfaction}
\label{sect:csp}
In recent years, a considerable amount of research concentrated
on the computational complexity of $\Csp(\Gamma)$ for \emph{finite} structures $\Gamma$. Feder and Vardi~\cite{FederVardi} conjectured that for such $\Gamma$, the
problem $\Csp(\Gamma)$ is either in P, or NP-complete\footnote{By Ladner's theorem~\cite{ladner}, there are infinitely many complexity classes between P and NP, unless P=NP.}.
This conjecture has been fascinating researchers from various areas,
for instance from graph theory~\cite{HellNesetrilSurvey} and from finite model theory~\cite{FederVardi,KolaitisVardi,AtseriasBulatovDawar}.
It has been discovered that complexity classification questions
translate to fundamental questions in universal algebra~\cite{JBK,IMMVW}, so that
lately also many researchers in universal algebra started to work on questions that directly correspond to questions about the complexity of CSPs.

For arbitrary infinite structures $\Gamma$ it can be shown that there
are problems $\Csp(\Gamma)$ that are in NP, but neither in P nor NP-complete, unless P=NP. In fact, it can be shown that for every computational problem $\mathcal P$ there is an infinite structure $\Gamma$ such that $\mathcal P$ and $\Csp(\Gamma)$ are equivalent under polynomial-time Turing reductions~\cite{BodirskyGrohe}.
However, there are several classes of
infinite structures $\Gamma$ for which the complexity
of $\Csp(\Gamma)$ can be classified completely.

In this section we will see three such classes of computational problems; they all have the property that
\begin{itemize}
\item every problem in this class can be formulated as $\Csp(\Gamma)$ where $\Gamma$ has a first-order definition in a \emph{base structure} $\Delta$;
\item $\Delta$ is ordered homogeneous Ramsey with finite signature.
\end{itemize}
For all three classes, the classification result can be obtained by the
same method, which we describe in the following two subsections.

\subsection{Climbing up the lattice}\label{sect:climb}
Clearly, if we add relations to a structure $\Gamma$ with a finite
relational signature, then the CSP of the structure thus obtained is computationally at least as complex as the CSP of $\Gamma$.
On the other hand,
when we add a primitive positive definable relation to $\Gamma$,
then the CSP of the resulting structure has a polynomial-time reduction
to $\Csp(\Gamma)$. This is not hard to show, and has been observed
for finite domain structures in~\cite{JeavonsClosure}; the same proof
also works for structures over an infinite domain.

\begin{lemma}\label{lem:pp-reduce}
Let $\Gamma = (D; R_1,\dots,R_l)$ be a relational structure,
and let $R$ be a relation that has a primitive positive definition in $\Gamma$. Then the problems $\Csp(\Gamma)$ and
$\Csp(D; R,R_1,\dots,R_l)$ are equivalent under polynomial-time reductions.
\end{lemma}

When we study the CSPs of the reducts $\Gamma$
of a structure $\Delta$,
we therefore consider the lattice of reducts of $\Delta$
which are closed under primitive positive definitions (i.e., which contain all relations that are primitive positive definable from the reduct), and describe the border between tractability and NP-completeness in this lattice.  We remark that the reducts of $\Delta$ have, since we expand them by all primitive positive definable relations, infinitely many relations, and hence do not define a CSP; however, we consider $\Gamma$ tractable if and only if all structures obtained from $\Gamma$ by dropping all but finitely many relations have a tractable CSP. Similarly, we consider $\Gamma$ hard if there exists a structure
obtained from $\Gamma$ by dropping all but finitely many relations
that has a hard CSP. With this convention, it is interesting to determine the \emph{maximal tractable} reducts, i.e., those reducts closed under primitive positive definitions which do not contain any hard relation and which cannot be further extended without losing this property.

Recall the notion of a \emph{clone} from Section~\ref{sect:reducts}. By Theorem~\ref{thm:clones-pp}, the lattice of primitive positive closed reducts of $\Delta$ and the lattice of closed clones containing $\Aut(\Delta)$ are antiisomorphic via the mappings $\Gamma\mapsto\Pol(\Gamma)$ (for reducts $\Gamma$) and ${\mathcal C}\mapsto \Inv({\mathcal C})$ (for clones ${\mathcal C}$). We refer to the introduction of \cite{BodChenPinsker} for a detailed exposition of this well-known connection. Therefore, the maximal tractable reducts correspond to \emph{minimal tractable} clones, which are precisely the clones of the form $\Pol(\Gamma)$ for a maximal tractable reduct.

The proof strategy of the classification results presented in Sections~\ref{ssect:equality}, \ref{ssect:temporal}, and \ref{ssect:schaefer} is as follows.
We start by proving that certain reducts $\Gamma$ have an NP-hard CSP. How to show this, and how to find those `basic hard reducts' will be the topic of the next subsection.
Let $R$ be one of the relations from those hard reducts.
If $R$ does not have a primitive positive definition in $\Gamma$, then Theorem~\ref{thm:clones-pp} implies that
$\Gamma$ has a polymorphism $f$ that does not preserve $R$.
We are now in a similar situation
as in Section~\ref{sect:minimalfunctions}.
Introducing constants, we can show that $f$ generates an operation $g$ that still does not preserve $R$ but
is canonical with respect to the expansion of $\Gamma$ by constants. There are only
finitely many canonical behaviours that $g$ might have, and therefore we can start a combinatorial
analysis.
In the three classifications that follow, this strategy always leads to polymorphisms
that imply that CSP$(\Gamma)$ can be solved in polynomial time.

\subsection{Primitive positive interpretations, and adding constants}
Surprisingly, in all the classification results that we present in Sections~\ref{ssect:equality}, \ref{ssect:temporal}, and \ref{ssect:schaefer}, there is a single condition
that implies that a CSP is NP-hard. 
Recall that an interpretation is called \emph{primitive positive} if all formulas involved in the interpretation (the domain formula, the formulas $\phi_R$ and $\phi_=$) are primitive positive.
The relevance of primitive positive interpretations in constraint
satisfaction comes from the following fact, which is known for finite domain constraint satisfaction, albeit not using
the terminology of primitive positive interpretations~\cite{JBK}.
In the present form, it appears first in~\cite{BodirskySurvey}.

\begin{theorem}\label{thm:pp-interpret-hard}
Let $\Gamma$ and $\Delta$ be structures with finite relational signatures.
If there is a primitive positive interpretation of $\Gamma$
in $\Delta$, then there is a polynomial-time reduction from
$\Csp(\Gamma)$ to $\Csp(\Delta)$.
\end{theorem}

All hardness proofs presented later can be shown
via primitive positive interpretations of Boolean structures (i.e., structures with the domain $\{0,1\}$)
with a hard CSP. In fact, in all such Boolean structures the relation $\NAE$ defined as
$$ \NAE = \{0,1\}^3 \setminus \{(0,0,0),(1,1,1)\}$$
 is primitive positive definable. This fact has not been stated
 in the original publications; however, it deserves to be mentioned
 as a unifying feature of all the classification results presented here.
It is often
more convenient to interpret other Boolean structures than
$(\{0,1\}; \NAE)$, and to then apply the following Lemma.
An operation $f: D^k \rightarrow D$ is called \emph{essentially a permutation} if there exists an $i$ and a bijection $g: DÊ\rightarrow D$ so that $f(x_1,\dots,x_k)=g(x_i)$ for all $(x_1,\dots,x_k) \in D^k$.

\begin{lemma}\label{lem:nae}
Let $\Delta$ be a structure that interprets a Boolean structure
$\Gamma$ such that
all polymorphisms of $\Gamma$ are essentially a permutation. Then the structure $(\{0,1\}; \NAE)$ has a primitive positive interpretation in $\Delta$, and $\Csp(\Delta)$ is NP-hard.
\end{lemma}
\begin{proof}
Since the polymorphisms of $\Gamma$ preserve the relation $\NAE$,
and by the well-known finite analog of Theorem~\ref{thm:clones-pp} (due to~\cite{Geiger} and independently, \cite{BoKaKoRo}),
$\NAE$ is primitive positive definable in $\Gamma$.
When $\phi$ is such a primitive positive definition,
by substituting all relations in $\phi$ by their
defining relations in $\Delta$ we obtain an interpretation of
$(\{0,1\}; \NAE)$ in $\Delta$.
Hardness of $\Csp(\Delta)$ follows from the NP-hardness of $\Csp((\{0,1\}; \NAE))$  (this problem is called \textsc{Not-all-3-equal-3Sat} in~\cite{GareyJohnson}) and Theorem~\ref{thm:pp-interpret-hard}.
\end{proof}

Typical Boolean structures $\Gamma$ such that all polymorphisms of
$\Gamma$ are essentially a permutation are the structure $(\{0,1\}; \{ (t_1,t_2,t_3,t_4) \in \{0,1\} \; | \; t_1+t_2+t_3+t_4=2\}$, the structure $(\{0,1\}; \OIT)$, or the structure $(\{0,1\}; \NAE)$ itself. 

Sometimes it is not possible to give a primitive positive interpretation of the structure $(\{0,1\}; \NAE)$ in $\Gamma$,
but it is possible after expanding $\Gamma$ with constants.
Under an assumption about the endomorphism monoid
of $\Gamma$, however, introducing constants does not
change the computational complexity of $\Gamma$.
More precisely, we have the following.

\begin{theorem}[Theorem 19 in~\cite{Cores-Journal}]\label{thm:constants-hard}
Let $\Gamma$ be an $\omega$-categorical structure
with a finite relational signature
such that $\Aut(\Gamma)$ is dense in $\End(\Gamma)$.
Then for any finite number
of elements $c_1,\dots,c_k$ of
$\Gamma$ there is a polynomial-time reduction
from $\Csp((\Gamma, \{c_1,\},\dots,\{c_k\}))$ to $\Csp(\Gamma)$.
\end{theorem}

\subsection{Reducts of equality}\label{ssect:equality}
One of the most fundamental classes of $\omega$-categorical
structures is the class of all reducts of $(X;=)$, where $X$ is an arbitrary countably infinite set.
Up to isomorphism,
this is exactly the class of countable structures that are preserved
by all permutations of their domain.
The other two classes of $\omega$-categorical structures that we will study here both contain this class.

We go straight to the statement of the complexity classification in terms of primitive positive interpretations.
This is essentially a reformulation of a result from~\cite{ecsps} which has been formulated without primitive positive interpretations.
It turns out that when $\Gamma$ is preserved
by the operations from one of the minimal clones above the clone generated by all the permutations of $X$, then CSP$(\Gamma)$ can be solved in polynomial time,
and otherwise CSP$(\Gamma)$ is NP-hard.


\begin{theorem}[essentially from~\cite{ecsps}]
Let $\Gamma$ be a reduct of $(X;=)$.
Then exactly one of the following holds.
\begin{itemize}
\item $\Gamma$ has a constant endomorphism. 
In this case, $\Csp(\Gamma)$ is trivially in P.
\item $\Gamma$ has a binary injective polymorphism. 
In this case, $\Csp(\Gamma)$ is in P.
\item All relations with a first-order definition in $(X;=)$ have a primitive positive definition in $\Gamma$. Furthermore,
the structure $(\{0,1\}; \NAE)$
has a primitive positive interpretation in $\Gamma$, and
$\Csp(\Gamma)$ is NP-complete.
\end{itemize}
\end{theorem}
\begin{proof}
It has been shown in~\cite{ecsps} that $\Csp(\Gamma)$ is in P
when $\Gamma$ has a constant or a binary injective polymorphism.
Otherwise, by Theorem~\ref{thm:minimalAboveS},
every polymorphism of $\Gamma$ is generated by the permutations of $X$.
Hence, every relation $R$ with a first-order definition
in  $(X;=)$ is preserved by all polymorphisms of $\Gamma$, and it follows from
Theorem~\ref{thm:clones-pp} that every relation is primitive positive definable in $\Gamma$.

This holds in particular for the relation $E_6$ defined as follows.
 \begin{align*} E_6 = \{(x_1,x_2,y_1,y_2,z_1,z_2) \in X^6\; | \; &
 (x_1=x_2 \wedge y_1 \neq y_2 \wedge z_1 \neq z_2) \\
   & \vee  \;
 (x_1 \neq x_2 \wedge y_1 = y_2 \wedge z_1 \neq z_2) \\
 & \vee  \; (x_1 \neq x_2 \wedge y_1 \neq y_2 \wedge z_1 = z_2) \} 
 \end{align*}
 We now show that the structure $(\{0,1\}; \OIT)$ has a primitive positive interpretation in $(X;E_6)$, 
 which by Lemma~\ref{lem:nae} also shows that 
 $(\{0,1\}; \NAE)$ has a primitive positive interpretation in $(X;E_6)$ and that $\Csp(\Gamma)$ is NP-hard.

The dimension of the interpretation is 2, and the domain formula
is \emph{`true'}.
The formula $\phi_{\OIT}(x_1,x_2,y_1,y_2,z_1,z_2)$ is $E_6(x_1,x_2,y_1,y_2,z_1,z_2)$, and
The formula $\phi_=(x_1,x_2,y_1,y_2)$ is
\begin{align*} \exists a_1,a_2,u_1,u_2,u_3,u_4,z_1,z_2. \; & a_1=a_2 \wedge
E_6(a_1,a_2,u_1,u_2,u_3,u_4) \\
\wedge \; & E_6(u_1,u_2,x_1,x_2,z_1,z_2)
\wedge E_6(u_3,u_4,z_1,z_2,y_1,y_2) .
\end{align*}
Note that the primitive positive formula $\phi_=(x_1,x_2,y_1,y_2)$ is
equivalent to $x_1=x_2 \Leftrightarrow y_1=y_2$.
The map $h$ maps
$(a_1,a_2)$ to $1$ if $a_1=a_2$, and to $0$ otherwise.
\end{proof}
Note that both the constant and the binary injective operation
are canonical as functions over $(X;=)$.

\subsection{Reducts of the dense linear order}\label{ssect:temporal}
An extension of the result in the previous subsection has been obtained
in~\cite{tcsps-journal}; there, the complexity of the CSP for all reducts of
$({\mathbb Q};<)$ has been classified. 
By a theorem of Cameron, those reducts are (again up to isomorphism) exactly
the structures that are \emph{highly set-transitive}~\cite{Cameron5}, i.e., structures $\Gamma$ such that for any two finite subsets $A,B$ with $|A|=|B|$ of the domain
there is an automorphism of $\Gamma$ that maps $A$ to $B$.

The corresponding class of CSPs contains many computational problems that have been
studied in Artificial Intelligence, in particular in temporal reasoning~\cite{vanBeek,Nebel,BroxvallJonsson},
but also in scheduling~\cite{and-or-scheduling} or general theoretical computer science~\cite{Opatrny,GalilMegiddo}.
The following theorem is a consequence of results from~\cite{tcsps-journal}. Again, we show that the hardness proofs in this class
are captured by interpreting Boolean structures with few
polymorphisms via primitive positive interpretations with finitely many parameters; this has not appeared in~\cite{tcsps-journal}, so
we provide the proof. The central arguments in the classification follow the reduct classification 
technique based on Ramsey theory that we present in this survey; see Figure~\ref{fig:tcsp} for
an illustration of the bottom of the lattice of reducts of $({\mathbb Q};<)$, and the border of tractability for such reducts.

\begin{figure}
\begin{center}
\includegraphics[scale=0.5]{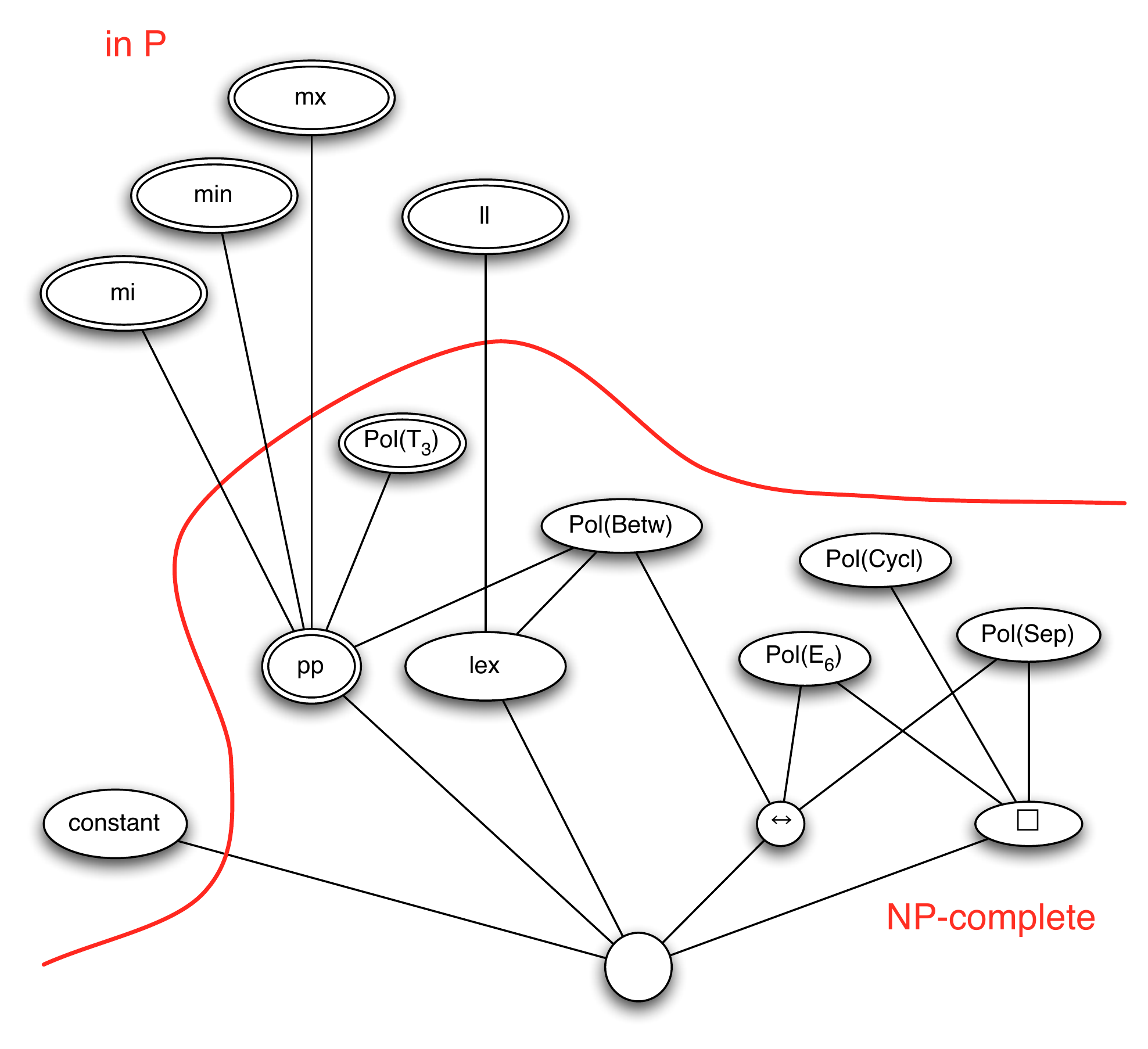}
\caption{An illustration of the classification result for temporal constraint languages. 
Double-circles mean that the corresponding operation has a dual generating a distinct clone which is not drawn in the figure. For the definition of mi, min, mx, and ll, see~\cite{tcsps-journal}.}
\label{fig:tcsp}
\end{center}
\end{figure}

\begin{theorem}[essentially from~\cite{tcsps-journal}]
\label{thm:tcsp}
Let $\Gamma$ be a reduct of
$({\mathbb Q};<)$. Then exactly one of the following holds.
\begin{itemize}
\item $\Gamma$ has one out of 9 binary polymorphisms 
(for a detailed description of those see~\cite{tcsps-journal}), and
$\Csp(\Gamma)$ is in P.
\item $\Aut(\Gamma)$ is dense in $\End(\Gamma)$,
and the structure $(\{0,1\}; \NAE)$
has a primitive positive
interpretation with finitely many parameters in $\Gamma$.
In this case, $\Csp(\Gamma)$ is NP-complete.
\end{itemize}
\end{theorem}

Before we derive Theorem~\ref{thm:tcsp} from what
has been shown in~\cite{tcsps-journal}, we would like
to point to Figure~\ref{fig:tcsp} for an illustration
of the clones that correspond to maximal tractable reducts.
The diagram also shows the constraint languages
that just contain one of the important relations $\Betw$ (introduced in the introduction), $\Cycl$, $\sep$ ($\Cycl$ and $\sep$ already appeared in Section~\ref{sect:reducts}), $E_6$ (which appeared earlier in this section), $T_3$, and $-T_3$. Here,
$T_3$ stands for the relation $$\{ (x,y,z) \in \mathbb Q^3\; | \; (x=y<z) \vee (x=z<y) \} \; , $$
and when $R \subseteq {\mathbb Q}^k$, then $-R$ denotes
$\{(-t_1,\dots,-t_k) \; | \; (t_1,\dots,t_k) \in R\}$. 

The importance of those relations comes from the fact (shown in~\cite{tcsps-journal}) that unless 
$\Gamma$ has one out of the 9 binary polymorphisms mentioned
in Theorem~\ref{thm:tcsp} then there is a primitive positive definition
of at least one of the relations $\btw$, $\Cycl$, $\sep$, $E_6$, $T_3$, or $-T_3$.

\begin{proof}[Proof of Theorem~\ref{thm:tcsp}]
It has been shown in~\cite{tcsps-journal} that unless
$\Gamma$ has a constant endomorphism,
$\Aut(\Gamma)$ is dense in $\End(\Gamma)$.
We have already seen that there is a primitive positive interpretation of
$(\{0,1\}; \NAE)$ in structures isomorphic to $({\mathbb Q}; E_6)$.

Now suppose that $T_3$ is primitive positive definable in $\Gamma$. We give below a primitive positive interpretation of the structure
$(\{0,1\}; \OIT)$ in $\Delta = ({\mathbb Q}; T_3,0)$.
Hence, there is also a primitive positive definition of
$(\{0,1\}; \OIT)$ in the expansion of $\Gamma$
by the constant $0$. 
Expansions by constants do not change the computational complexity
of $\Csp(\Gamma)$ since
$\Aut(\Gamma)$ is dense in $\End(\Gamma)$.
Thus, Lemma~\ref{lem:nae} shows NP-hardness of $\Csp(\Gamma)$, and that $(\{0,1\}; \NAE)$ has a primitive positive interpretation 
in $(\Gamma,0)$.

The interpretation of $(\{0,1\}; \OIT)$ in $\Delta$
\begin{itemize}
\item has dimension 2;
\item the domain formula
$\delta(x_1,x_2)$ is  $T_3(0,x_1,x_2)$;
\item the formula $\phi_{\OIT}(x_1,x_2,y_1,y_2,z_1,z_2)$ is
$$ \exists u. \; T_3(u,x_1,y_1) \wedge T_3(0,u,z_1) \; ;$$
\item the formula $\phi_=(x_1,x_2,y_1,y_2)$ is $T_3(0,x_1,y_2)$;
%
\item the coordinate map $h: \delta(\Delta^2) \rightarrow \{0,1\}$
is defined as follows. Let $(b_1,b_2)$ be a pair of elements
of $\Delta$ that satisfies $\delta$.
Then exactly one of $b_1, b_2$ must have value $0$,
and the other element is strictly greater than $0$.
We define $h(b_1,b_2)$ to be $1$ if $b_1=0$,
and to be $0$ otherwise. 
\end{itemize}
To see that this is the intended interpretation, let
$(x_1,x_2),(y_1,y_2),(z_1,z_2) \in \delta(\Delta^2)$, and 
suppose that 
$t:=(h(x_1,x_2),h(y_1,y_2),h(z_1,z_2))=(1,0,0) \in \OIT$.
We have to verify that $(x_1,x_2,y_1,y_2,z_1,z_2)$ satisfies $\phi_{\OIT}$ in $\Delta$.
Since $h(x_1,x_2)=1$, we have $x_1 = 0$, and similarly we get
that $y_1,z_1 > 0$. We can then set
$u$ to $0$ and have $T_3(u,x_1,y_1)$ since $0=u=x_1<y_1$,
and we also have $T_3(0,u,z_1)$ since $0=u<z_1$. 
The case that $t=(0,1,0)$ is analogous. Suppose now that 
$t=(0,0,1) \in \OIT$.
Then $x_1,y_1 > 0$, and $z_1=0$. We can then set $u$ to
$\min(x_1,y_1)$, and therefore have $T_3(u,x_1,y_1)$,
and $T_3(0,u,z_1)$ since $0=z_1<u$.
Conversely, suppose that $(x_1,x_2,y_1,y_2,z_1,z_2)$ satisfies $\phi_{\OIT}$ in $\Delta$. Since $T_3(0,u,z_1)$, exactly one out of 
$u,z_1$ equals $0$. When $u=0$, then because of
$T_3(u,x_1,y_1)$ exactly one out of $x_1,y_1$ equals $0$,
and we get that $(h(x_1,x_2),h(y_1,y_2),h(z_1,z_2)) \in \{(0,1,0),(1,0,0)\} \subseteq \OIT$. When $u>0$, then $x_1>0$ and $y_1>0$,
and so $(h(x_1,x_2),h(y_1,y_2),h(z_1,z_2)) = (0,0,1) \in \OIT$.

An interpretation of $(\{0,1\}; \OIT)$ in $({\mathbb Q}; -T_3,0)$
can be obtained in a dual way.

Next, suppose that $\Betw$ is primitive positive definable in $\Gamma$. We will give a primitive positive interpretation of $(\{0,1\}; \NAE)$ in $({\mathbb Q}; \Betw, 0)$.
Hence, when $\Betw$ has a primitive positive definition in $\Gamma$,
then by Theorem~\ref{thm:constants-hard} (since
$\Aut(\Gamma)$ is dense in $\End(\Gamma)$)
and Lemma~\ref{lem:nae} we obtain NP-hardness of $\Csp(\Gamma)$.

The dimension of the interpretation is one, and the domain formula is $x \neq 0$, which is
clearly equivalent to a primitive positive formula over $({\mathbb Q}; \Betw,0)$. 
The map $h$ maps positive points to $1$, and all other points from ${\mathbb Q}$ to $0$.
The formula $\phi_=(x_1,y_1)$ is
$$ \exists z. \; \Betw(x_1,0,z) \wedge \Betw(z,0,y_1)$$
Note that the primitive positive formula $\phi_=$
 is over $({\mathbb Q}; \Betw, 0)$
 equivalent to $(x_1>0 \Leftrightarrow y_1 > 0)$.
Finally, $\phi_{\NAE}(x_1,y_1,z_1)$
is
$$ \exists u. \; \Betw(x_1,u,y_1) \wedge \Betw(u,0,z_1) \; .$$

If $\sep$ has a primitive positive definition in $\Gamma$,
then the statement follows easily from the previous argument
since $\Betw(x,y,z)$ has a 1-dimensional primitive positive
interpretation in $({\mathbb Q}; \sep)$
(the formula $\phi_\Betw(x,y,z)$ is $\exists u. \sep(u,x,y,z)$).

Finally, if $\Cycl$ is primitive positive definable in $\Gamma$,
we give a 3-dimensional primitive positive interpretation
of the structure $(\{0,1\}; R,\neg)$ where
$R = \{0,1\}^3 \setminus \{(0,0,0)\}$ and $\neg = \{(0,1),(1,0)\}$.
The idea of the interpretation is inspired by the NP-hardness proof of~\cite{GalilMegiddo} for the `Cyclic ordering problem' (see~\cite{GareyJohnson}).


The dimension of our interpretation is three, and the domain formula $\delta(x_1,x_2,x_3)$ is $x_1 \neq x_2 \wedge x_2 \neq x_3 \wedge x_3 \neq x_1$,
which clearly has a primitive positive definition in $({\mathbb Q}; \Cycl)$.
The coordinate map $h$ sends $(x_1,x_2,x_3)$ to $0$ if
$\Cycl(x_1,x_2,x_3)$, and to $1$ otherwise. 

Let $\phi(x_1,x_2,x_3,y_1,y_2,y_3)$ be the formula
\begin{align*}
& \Cycl(x_1,y_1,x_2) \wedge \Cycl(y_1,x_2,y_2) \wedge \Cycl(x_2,y_2,x_3) \\
& \Cycl(y_2,x_3,y_3) \wedge \Cycl(x_3,y_3,x_1) \wedge \Cycl(y_3,x_1,y_1)  \; .
\end{align*}
When $(a_1,\dots,a_6)$ satisfies $\phi$, we can imagine $a_1,\dots,a_6$ as points that appear clockwise in this order on the unit circle.
In particular, we then have that $\Cycl(a_1,a_3,a_5)$ holds if and only if $\Cycl(a_2,a_4,a_6)$ holds. 
The formula $\phi_=(x_1,x_2,x_3,y_1,y_2,y_3)$ is
\begin{align*}
\exists u^1_1,\dots,u^4_3. \; & \phi(x_1,x_2,x_3,u^1_1,u^1_2,u^1_3) \wedge \\
& \bigwedge_{i=1}^3 \phi(u^i_1,u^i_2,u^i_3,u^{i+1}_1,u^{i+1}_2,u^{i+1}_3) \wedge \phi(u^4_1,u^4_2,u^4_3,y_1,y_2,y_3) \, ,
\end{align*}
which is equivalent to $$\delta(x_1,x_2,x_3) \wedge \delta(y_1,y_2,y_3) \wedge (\Cycl(x_1,x_2,x_3) \Leftrightarrow \Cycl(y_1,y_2,y_3)) \, ;$$ this is tedious, but straightforward to verify, and we omit the proof.

The formula $\phi_\neg(x_1,x_2,x_3,y_1,y_2,y_3)$
is $\phi_=(x_1,x_2,x_3,z_1,z_3,z_2)$.

The formula
$\phi_{R}(x_1,x_2,x_3,y_1,y_2,y_3,z_1,z_2,z_3)$ is
\begin{align*}
\exists a,b,c,d,e,f,g,h,i,j,k,l,m,n. \; & \Cycl(a,c,j) \wedge \Cycl(b,j,k) \wedge \Cycl(c,k,l) \\
\wedge \; & \Cycl(d,f,j) \wedge \Cycl(e,j,l) \wedge \Cycl(f,l,m) \\
\wedge \; & \Cycl(g,i,k) \wedge \Cycl(h,k,m) \wedge \Cycl(i,m,n) \\\wedge  \; & \Cycl(n,m,l)
\wedge \phi_=(x_1,x_2,x_3,a,b,c) \\
\wedge  \; & \phi_=(y_1,y_2,y_3,d,e,f) \wedge \phi_=(z_1,z_2,z_3,g,h,i)
\end{align*}
The proof that for all tuples $\bar a_1, \bar a_2, \bar a_3 \in
{\mathbb Q}^3$
\begin{align*}
(h(\overline a_1), h(\overline a_3),  h(\overline a_3)) \in R \;
& \Leftrightarrow \;
({\mathbb Q}; \Cycl) \models \phi_{R}(\overline a_1, \overline a_2, \overline a_3)
\end{align*}
follows directly the correctness proof of the reduction presented
in~\cite{GalilMegiddo}. 
\end{proof}

\subsection{Reducts of the random graph}\label{ssect:schaefer}
The full power of the technique that is developed in this paper
can be used to obtain a full complexity classification for
all reducts of the random graph $G=(V;E)$~\cite{BodPin-Schaefer}. Again, the result can be stated in terms of primitive positive
interpretations -- this is not obvious from the statement of the result
in~\cite{BodPin-Schaefer}, therefore we provide the proofs.

\begin{theorem}[essentially from~\cite{BodPin-Schaefer}]\label{thm:schaefer}
Let $\Gamma$ be a reduct of
the countably infinite random graph $G$.
Then exactly one of the following holds.
\begin{itemize}
\item $\Gamma$ has one out of 17 at most ternary canonical
polymorphisms (for a detailed description of those see~\cite{BodPin-Schaefer}), and $\Csp(\Gamma)$ is in P.
\item $\Gamma$ admits a primitive positive interpretation of $(\{0,1\}; \OIT)$. In this case, $\Csp(\Gamma)$ is NP-complete.
\end{itemize}
\end{theorem}

\begin{proof}
It has been shown in~\cite{BodPin-Schaefer} that
$\Gamma$ has one out of 17 at most ternary canonical polymorphisms, and $\Csp(\Gamma)$ is in P,
or one of the following relations has a
primitive positive definition in $\Gamma$:
the relation $E_6$, or the relation $T$, $H$, or $P^{(3)}$,
which are defined as follows.
The $4$-ary relation $T$ holds on $x_1,x_2,x_3,x_4 \in V$
if $x_1,x_2,x_3,x_4$ are pairwise distinct, and induce in $G$
either
\begin{itemize}
\item a single edge and two isolated vertices,
\item a path with two edges and an isolated vertex,
\item a path with three edges, or
\item a complement of one of the structures stated above.
\end{itemize}
To define the relation $H$, we write $N(u,v)$ as a shortcut for
$E(u,v) \wedge u \neq v$. Then
$H(x_1,y_1,x_2,y_2,x_3,y_3)$ holds on $V$ if
\begin{align*}
& \bigwedge_{i,j \in \{1,2,3\}, i \neq j, u \in \{x_i,y_i\},v \in \{x_j,y_j\}} N(u,v) \\
\wedge & \; \big(((E(x_1,y_1) \wedge N(x_2,y_2) \wedge N(x_3,y_3))
\\
& \vee \; (N(x_1,y_1) \wedge E(x_2,y_2) \wedge N(x_3,y_3))   \label{eq:rel} \\
& \vee \; (N(x_1,y_1) \wedge N(x_2,y_2) \wedge E(x_3,y_3)) \big)\; .
\end{align*}
The ternary relation $P^{(3)}$ holds on $x_1,x_2,x_3$ if those
three vertices are pairwise distinct and do not induce a clique or an independent set in $G$.

Suppose first that $T$ is primitive positive definable in $\Gamma$.
Let $R$ be the relation
$\{ (t_1,t_2,t_3,t_4) \in \{0,1\} \; | \; t_1+t_2+t_3+t_4=2\}$.
We have already mentioned that all polymorphisms of 
$(\{0,1\}; R)$ are essentially permutations.
To show that $(\{0,1\}; \NAE)$ has a primitive positive
interpretation in $\Gamma$, we can therefore use
Lemma~\ref{lem:nae} and it suffices to show that
 there is a primitive positive interpretation of the structure
$(\{0,1\}; R)$ in $(V;T)$.
For a finite subset $S$ of $V$, write $\# S$ for the parity of edges
 between members of $S$. Now we define the relation $L \subseteq V^6$ as follows.
\begin{align*}
L := \big \{ x \in V^6 \; | \; & \text{the entries of $x$ are pairwise distinct, and }\\
&\#\{x_1,x_2,x_3\}=\#\{x_4,x_5,x_6\} \big \} 
\end{align*}
It has been shown in~\cite{BodPin-Schaefer}
that the relation $L$ is pp-definable in $(V;T)$.
We therefore freely use the relation
$L$ (and similarly $\neq$, the disequality relation) in primitive positive formulas over $(V;T)$.

Our primitive positive interpretation of
$(\{0,1\}; R)$ has dimension three.
The domain formula $\delta(x_1,x_2,x_3)$ is
$x_1 \neq x_2 \; \wedge \; x_1 \neq x_3 \; \wedge \; x_2 \neq x_3$.
The formula $\phi_R(x^1_1,x^1_2,x^1_3,\dots,x^4_1,x^4_2,x^4_3)$ of the interpretation is
\begin{align*}
\exists y_1,y_2,y_3,y_4. \; & T(y_1,\dots,y_4) \\
\wedge & L(x^1_1,x^1_2,x^1_3,y_2,y_3,y_4) \\
\wedge & L(x^2_1,x^2_2,x^2_3,y_1,y_3,y_4) \\
\wedge & L(x^3_1,x^3_2,x^3_3,y_1,y_2,y_4) \\
\wedge & L(x^4_1,x^4_2,x^4_3,y_1,y_2,y_3)
\end{align*}
The formula $\phi_=$ is $L(x_1,x_2,x_3,y_1,y_2,y_3)$.
Finally, the coordinate map sends a tuple $(a_1,a_2,a_3)$
for pairwise distinct $a_1,a_2,a_3$
to $1$ if
$P^{(3)}(a_1,a_2,a_3)$, and to $0$ otherwise.

\ignore{
Observe that when $(a^1_1,a^1_2,a^1_3)$ and
 $(a^2_1,a^2_2,a^2_3)$ are tuples of pairwise distinct elements,
 then
$h(a^1_1,a^1_2,a^1_3) = h(a^2_1,a^2_2,a^2_3)$
if and only if $L(a^1_1,a^1_2,a^1_3,a^2_1,a^2_2,a^2_3)$,
by the definition of $L$.
We have to show that
$(h(a^1_1,a^1_2,a^1_3),\dots,h(a^4_1,a^4_2,a^4_3)) \in R$ if and only
if $\phi_R(a^1_1,a^1_2,a^1_3,\dots,a^4_1,a^4_2,a^4_3)$,
for tuples $(a^i_1,a^i_2,a^i_3)$, $1 \leq i \leq 4$, with pairwise distinct entries. This follows from the
observation that for any tuple $t \subseteq V^4$ that satisfies
$T(x_1,x_2,x_3,x_4)$,
 there are exactly two 3-element subsets $\{a,b,c\}$ of
$\{x_1,x_2,x_3,x_4\}$ such that $t$ satisfies the formula $R^{(3)}(a,b,c)$. So, by the four constraints of $\phi_R$ involving the relation $L$,
$\phi_R(a^1_1,a^1_2,a^1_3,\dots,a^4_1,a^4_2,a^4_3)$ holds if
and only if there are exactly two out of the four triples
$(a^1_1,a^1_2,a^1_3)$, $(a^2_1,a^2_2,a^2_3)$, $(a^3_1,a^3_2,a^3_3)$, $(a^4_1,a^4_2,a^4_3)$ that lie in $R^{(3)}$,
which is the case if and only if
$(h(a^1_1,a^1_2,a^1_3),\dots,h(a^4_1,a^4_2,a^4_3) \in R$.
}

Next, suppose that $H$ is primitive positive definable in $\Gamma$.
We give a 2-dimensional interpretation of $(\{0,1\}; \OIT)$
in $\Gamma$. The domain formula is \emph{`true'}. The formula
$\phi_=(x_1,x_2,y_1,y_2)$ is
\begin{align*}
\exists z_1,z_2,u_1,u_2,v_1,v_2. \; & H(x_1,x_2,u_1,u_2,z_1,z_2) \wedge N(u_1,u_2) \\
\wedge \; & H(z_1,z_2,v_1,v_2,y_1,y_2) \wedge N(v_1,v_2) \, .
\end{align*}
This formula is equivalent to a primitive positive formula
over $\Gamma$ since $N(x,y)$ is primitive positive definable
by $H$. The formula $\phi_{\OIT}(x_1,x_2,y_1,y_2,z_1,z_2)$
is
\begin{align*}
\exists x_1',x_2',y_1',y_2',& z_1',z_2'. \,
H(x_1',x_2',y_1',y_2',z_1',z_2') \\
\wedge \; & \phi_=(x_1,x_2,x_1',x_2')
\wedge \phi_=(x_1,x_2,x_1',x_2')
\wedge \phi_=(x_1,x_2,x_1',x_2') \, .
\end{align*}
The coordinate map sends a tuple $(x_1,x_2)$ to $1$ if $E(x_1,x_2)$
and to $0$ otherwise.

Finally, suppose that $P^{(3)}$ has a primitive positive definition
in $\Gamma$. We give a 2-dimensional primitive positive
interpretation of $(\{0,1\}; \NAE)$.
For $k\geq 3$, let $Q^{(k)}$ be the $k$-ary relation
that holds for a tuple $(x_1,\dots,x_k) \in V^k$
iff $x_1,\dots,x_k$ are pairwise distinct, and $(x_1,\dots,x_k)\nin P^{(k)}$. It has been shown in~\cite{BodPin-Schaefer} that the relation $Q^{(4)}$
is primitive positive definable by the relation $P^{(3)}$.
Now, the formula $\phi_=(x_1,x_2,y_1,y_2)$ is
$\exists z_1,z_2. Q^{(4)}(x_1,x_2,z_1,z_2) \wedge Q^{(4)}(z_1,z_2,y_1,y_2)$. The formula $\phi_{\NAE}(x_1,x_2,y_1,y_2,z_1,z_2)$ is
\begin{align*}
\exists u,v,w. \; & P^{(3)}(u,v,w) \wedge Q^{(4)}(x_1,x_2,u,v) \\
\wedgeÊ\; & Q^{(4)}(y_1,y_2,v,w) \wedge Q^{(4)}(z_1,z_2,w,u) \, .
\end{align*}
The coordinate map sends a tuple $(x_1,x_2)$ to $1$ if $E(x_1,x_2)$
and to $0$ otherwise.
\end{proof}

\section{Concluding Remarks and Further Directions}
We have outlined an approach to use Ramsey theory for the classification of reducts of a structure,
considered up to existential positive, or primitive positive interdefinability.
The central idea in this approach is to study functions that preserve the reduct, and to apply
structural Ramsey theory to show that those functions must act regularly on large parts of the domain.
This insight makes those functions accessible to combinatoral arguments and classification.

Our approach has been illustrated for the reducts of $(\mathbb Q; <)$, and the reducts of the random graph $(V;E)$.
One application of the results is complexity classification of constraint satisfaction problems
in theoretical computer science.
Interestingly, the hardness proofs in those classifications all follow a common pattern:
they are based on primitive positive interpretations.
In particular, we proved complete complexity classifications without the typical
 computer science hardness proofs -- rather, the hardness results follow from mathematical statements about primitive positive interpretability in $\omega$-categorical structures.

There are many other natural and important $\omega$-categorical structures besides $(\mathbb Q; <)$
and $(V;E)$ where this approach seems promising.
We have listed some of the simplest and most basic examples in Figure~\ref{fig:table}.
In this table, the first column specifies the `base structure' $\Delta$, and we will be interested in the
class of all structures definable in $\Delta$. The second column lists what is known about this class, considered up to first-order interdefinability. The third column describes the corresponding Ramsey result, when $\Delta$ is equipped with an appropriate linear order. The fourth column gives the status with respect to complexity classification of the corresponding class of CSPs. The fifth class indicates in which areas in computer science those CSPs find applications.

\begin{figure}[h]
\begin{center}
\begin{tabular}{|p{2.4cm}|p{2.5cm}|p{1.9cm}|p{1.6cm}|p{2cm}|}
\hline
Reducts of & First-order Reducts & Ramsey Class & CSP Dichotomy & Application, Motivation \\
 \hline \hline
$(X; =)$ & Trivial & Ramsey's theorem & Yes & Equality Constraints \\
\hline
$(\mathbb Q; <)$ & Cameron~\cite{Cameron} & Ramsey's theorem & Yes & Temporal Reasoning \\
\hline
$(V; E)$ & Thomas~\cite{RandomReducts} & Ne\v{s}et\v{r}il + R\"odl~\cite{NesetrilRoedl} & Yes & Schaefer's theorem for graphs \\
 \hline
 Homogeneous universal poset & ? & Ne\v{s}et\v{r}il + R\"odl~\cite{PosetRamsey} & ? & Temporal Reasoning \\
\hline
Homogeneous C-relation & Adeleke, Macpherson, Neumann~\cite{AdelekeMacPherson,AdelekeNeumann} & Deuber, Miliken & ? & Phylogeny Reconstruction \\
 \hline
 Countable atomless Boolean algebra & ? & Graham, Leeb, Rothschild (see~\cite{Topo-Dynamics}) & ? & Set Constraints \\
  \hline
 Allen's Interval Algebra  & ? & This paper, Section~\ref{sect:interpret} & ? & Temporal Reasoning \\
 \hline
\end{tabular}
\end{center}
\caption{A diagram suggesting future research.}
\label{fig:table}
\end{figure}


\end{document}